\documentclass[12pt,a4paper,dvipsnames]{article}

\usepackage[utf8]{inputenc}

\usepackage[english]{babel}
\usepackage{amsmath}
\usepackage{amsthm}
\usepackage{amssymb}
\usepackage{geometry}
\usepackage{a4wide}
\usepackage{bm}
\usepackage{microtype}
\usepackage{mathtools}
\usepackage{csquotes}

\usepackage{subcaption}

\usepackage{enumitem}
\setlist[enumerate]{itemsep=0mm,parsep=2mm,topsep=5pt,label=\textit{(\alph*)}}

\usepackage{xcolor}
\newcommand\myshade{85}
\colorlet{myurlcolor}{Aquamarine}
\usepackage{hyperref}
\hypersetup{
  linkcolor  = black,
  citecolor  = black,
  urlcolor   = myurlcolor!\myshade!black,
  colorlinks = true,
}
\usepackage[capitalise, noabbrev, nameinlink, nosort]{cleveref}
\crefname{equation}{}{}
\crefname{appsec}{Appendix}{Appendices}

\usepackage{graphicx, standalone, caption, wrapfig}
\usepackage{pgf,tikz,pgfplots}
\pgfplotsset{compat=1.17}
\usetikzlibrary{shapes.geometric}
\usetikzlibrary{patterns}
\usetikzlibrary{decorations.pathreplacing}
\usepackage{tkz-euclide}
\usetikzlibrary{external}
\definecolor{edgeblack}{rgb}{0.1,0.1,0.1}
\definecolor{vertexblack}{rgb}{0.05,0.05,0.05}

\theoremstyle{definition}
\newtheorem{theorem}{Theorem}[section]

\newtheorem{lemma}[theorem]{Lemma}
\newtheorem*{lemma*}{Lemma}
\newtheorem*{conjecture*}{Conjecture}
\newtheorem*{lemma''*}{``Lemma''}

\newtheorem*{claim*}{Claim}
\newtheorem{corollary}[theorem]{Corollary}

\newtheorem{conjecture}[theorem]{Conjecture}

\newcommand{\RR}{\mathbb{R}}

\newcommand{\QQ}{\mathbb{Q}}

\newcommand{\sep}[1]{R_{#1}}
\newcommand{\reflection}[1]{\mathcal{R}_{#1}}
\newcommand{\partialeq}{\sim}

\newcommand{\glclosure}{\text{cl}^*_d}

\def\mainclass[#1]{$#1$-joined}

\newcommand{\norm}[1]{\left\lVert#1\right\rVert}

\DeclareMathOperator{\dom}{dom}
\newcommand{\Rd}{\mathcal{R}_d}

\usepackage{authblk}

\title{Partial reflections and globally linked pairs in rigid graphs}
\author[1,2]{Dániel Garamvölgyi} 
\author[2,3]{Tibor Jordán}
\affil[1]{HUN-REN Alfréd Rényi Institute of Mathematics, Reáltanoda utca 13-15, Budapest, 1053, Hungary}
\affil[2]{HUN-REN--ELTE Egerváry Research Group on Combinatorial Optimization, Pázmány~Péter~sétány~1/C, Budapest, 1117, Hungary}
\affil[3]{Department of Operations Research, ELTE Eötvös Loránd University, Pázmány~Péter sétány~1/C, Budapest, 1117, Hungary}
\affil[ ]{\footnotesize \textit{E-mail addresses:} {\tt \{daniel.garamvolgyi,tibor.jordan\}@ttk.elte.hu}}

\begin{document}

\maketitle

\begin{abstract}
A $d$-dimensional framework is a pair $(G,p)$, where $G$ is a graph and $p$ maps the vertices of $G$ to points in $\RR^d$. The edges of $G$ are mapped to the corresponding line segments. A graph $G$ is said to be globally rigid in $\RR^d$ if every generic $d$-dimensional framework $(G,p)$ is determined, up to congruence, by its edge lengths. A finer property
is global linkedness: we say that a vertex pair $\{u,v\}$ of $G$ is globally linked in $G$ in $\RR^d$ if in every generic $d$-dimensional framework $(G,p)$ the distance between $u$ and $v$ is uniquely determined by the edge lengths.

In this paper we investigate globally linked pairs in graphs in $\RR^d$. We give several characterizations of those rigid graphs $G$ in which a pair $\{u,v\}$ is globally linked if and only if there exist $d+1$ internally disjoint paths from $u$ to $v$ in $G$. We call these graphs $d$-joined. 
Among others, we show that $G$ is $d$-joined if and only if for each pair of generic frameworks of $G$ with the same edge lengths, one can be obtained from the other by a sequence of partial reflections along hyperplanes determined by $d$-separators of $G$. We also show that the family of $d$-joined graphs is closed under edge addition, as well as under gluing along $d$ or more vertices. As a key ingredient to our main results, we prove that rigid graphs in $\RR^d$ contain no crossing $d$-separators.

Our results give rise to new families of graphs for which global linkedness (and global rigidity) in $\RR^d$ can be tested in polynomial time. 
\end{abstract}

\section{Introduction}\label{section:introduction}

One of the major unsolved problems in rigidity theory is finding a combinatorial characterization for globally rigid graphs in $\RR^d$, for $d \geq 3$. We say that a graph $G$ is \textit{globally rigid in $\RR^d$} if every generic $d$-dimensional realization of $G$ is determined, up to congruence, by the Euclidean lengths of the edges in the realization. Here a \textit{realization} of $G$, or a (bar-and-joint) \textit{framework},  is a pair $(G,p)$, where $p: V(G) \rightarrow \RR^d$ is an embedding of the vertex set of $G$ into Euclidean space. The realization is \emph{generic} if the multiset of the $d \cdot |V(G)|$ coordinates of $\{p(v), v\in V\}$ is algebraically independent over $\QQ$. (We shall give detailed definitions and references in \cref{section:preliminaries}.) There has been much progress on globally rigid graphs in recent years, including a combinatorial characterization of the $d = 2$ case \cite{jackson.jordan_2005}, combinatorial characterizations for restricted classes of graphs and frameworks, e.g., \cite{jordan.etal_2016}, as well as a linear algebraic characterization for every $d \geq 1$ \cite{connelly_2005,gortler.etal_2010}. 
However, solving the general problem, or even formulating a plausible conjecture still seems challenging.

A promising avenue is the investigation of globally linked pairs. Following \cite{jackson.etal_2006}, we define a pair of vertices $\{u,v\}$ to be \emph{globally linked} in $G$ in $\RR^d$ if in every generic $d$-dimensional realization of $G$, the distance between $u$ and $v$ is uniquely determined by the edge lengths of the realization. That is, if $(G,p)$ is a generic realization and $(G,q)$ is another $d$-dimensional realization in which the length of each edge is the same, then the distance between $p(u)$ and $p(v)$ is equal to the distance between $q(u)$ and $q(v)$. It is immediate from the definitions that a graph is globally rigid in $\RR^d$ if and only if every pair of vertices is globally linked in the graph in $\RR^d$. In this sense the property of ``being globally linked'' gives a finer combinatorial structure than global rigidity. It is not surprising, then, that characterizing globally linked pairs seems even more difficult than characterizing global rigidity, and indeed, the former is open already in the $d=2$ case. However, apart from being an important research problem in its own right, we believe that even partial results may shed new light on the combinatorial nature of global rigidity.
Characterizations for some graph classes and a conjecture for the complete characterization are available in the $d=2$ case \cite{jackson.etal_2006,jackson.etal_2014,conferencepaper}. In higher dimensions the only result we are aware of is a
description of globally linked pairs in braced triangulations in $\RR^3$ \cite{jordan.tanigawa_2019}.

The above-mentioned linear algebraic characterization of global rigidity implies the fundamental fact that global rigidity is a ``generic property'': either every generic realization of the graph is uniquely determined by its edge lengths, or none of them are. In contrast, it is well-known that ``being globally linked in $\RR^d$'' is not a generic property in general, see \cite[Figs. 1 and 2]{jackson.etal_2006}; but interestingly, it is generic in almost every graph family where a characterization is known. This observation motivates the study of graphs $G$ that satisfy the following. 

\medskip
\noindent \textit{(a)} {\it ``Being globally linked in $G$ in $\RR^d$'' is a generic property.}
\medskip

In fact, in almost all known cases the following property characterizes globally linked pairs.

\medskip
\noindent \textit{(b)} {\it A nonadjacent pair of vertices $\{u,v\}$ is globally linked in $G$ in $\RR^d$ if and only if there are at least $d+1$ internally vertex-disjoint paths between $u$ and $v$ in $G$.}
\medskip

By Menger's theorem, if there are at most $d$ internally vertex-disjoint paths between $u$ and $v$, then there exists a separator $S$ of size at most $d$ in $G$ such that $u$ and $v$ are in different components $C_u$, $C_v$ of $G-S$.  It is a classical observation noted by Hendrickson\footnote{Hendrickson already refers to this as ``well-known'', but we are not aware of earlier references.} \cite{hendrickson_1992} that the existence of such a small separator lets us apply a \emph{partial reflection} to each generic realization of the graph. That is, given a generic realization $(G,p)$, we can reflect the points $p(w), w \in C_v$ in an affine hyperplane containing the points $p(s), s \in S$. In this way, we obtain another realization with the same edge lengths, but in which the distance between $u$ and $v$ is different. This shows that property \textit{(b)} implies \textit{(a)}. We note that the converse is not true: for example, it follows from \cite[Theorem 7.1]{jackson.etal_2006} that the complete bipartite graph $K_{3,3}$ satisfies \textit{(a)}, but not \textit{(b)}, in the two-dimensional case.

Our main goal in this paper, and the content of \cref{section:3char}, is the investigation of graphs with property \textit{(b)}. More precisely, we define a graph to be \mainclass[d] if it has at least $d+1$ vertices, it is rigid in $\RR^d$ and it satisfies \textit{(b)}. The condition on the number of vertices is there to exclude trivial cases, while the condition that the graph is rigid seems natural in the context of studying global rigidity. Nonetheless, we show that replacing rigidity by a weaker condition related to the structure of separators in the graph leads to the same graph family (\cref{corollary:weakerdef}).
Although finding a combinatorial characterization of \mainclass[d] graphs remains an open problem (apart from the $d=1$ case, which is easy), we give several geometric characterizations. In particular, we show that if $G$ is \mainclass[d] and $(G,p)$ is a generic $d$-dimensional realization, then every $d$-dimensional realization of $G$ with the same edge lengths as $(G,p)$ can be obtained from $(G,p)$ by a sequence of partial reflections, possibly followed by a congruence (\cref{theorem:partialreflectionchar}). Using this result we show that the family of \mainclass[d] graphs is closed under edge addition (\cref{theorem:monotonicity}), as well as under the operation of gluing along at least $d$ vertices (\cref{theorem:maingluing}). 

To prove these statements we need a careful analysis of the action of (sequences of) partial reflections on generic frameworks. We provide such an analysis in  \cref{section:partialreflections}, which may also be useful in other problems concerning rigidity and global rigidity of graphs and frameworks. We stress that most of our results regarding partial reflections are about generic frameworks, and that in the non-generic setting the action of partial reflections seems to be much more complicated, see, e.g., \cref{fig:nongeneric}. In order to apply our results to rigid graphs we also show, for all $d \geq 1$, that rigid graphs in $\RR^d$ contain no ``crossing $d$-separators'' (\cref{theorem:rigidnoncrossing}), extending a result of \cite{jackson.jordan_2005} from $d=2$ to the general case. 

Whenever we can show that a family of graphs is \mainclass[d], global linkedness and global rigidity 
become easy to verify and test by algorithms for computing local connectivity, not just for each graph in the family but also, by monotonicity, for their supergraphs.
To obtain a large class of \mainclass[d] graphs, we introduce the family of graphs with a {\it globally rigid gluing construction} in $\RR^d$. Roughly speaking, this notion means that the graph can be constructed from globally rigid graphs in $\RR^d$ by a sequence of gluing operations on pairs of graphs, where each operation
is performed along $d$ or more vertices.
We show that graphs that can be constructed in this way are \mainclass[d] (\cref{theorem:gluingconstruction}), and that
the family of these graphs is also closed under adding edges (\cref{lemma:gluingmonotonicity}). As an immediate corollary, we obtain that a graph that has a globally rigid gluing construction in $\RR^d$ is globally rigid in $\RR^d$ if and only if it is $(d+1)$-connected.

We also show that maximal outerplanar graphs, or more generally, $d$-connected 
chordal graphs possess a globally rigid gluing construction in $\RR^2$ (resp.\ $\RR^d$), so they are \mainclass[2] (resp.\ \mainclass[d]). Moreover, the monotonicity property implies that their ``braced'' versions, obtained by adding a set of new edges, are also \mainclass[2] (resp.\ \mainclass[d]). These corollaries extend the results of \cite{conferencepaper} and provide further classes of graphs in which the existence of a well-structured spanning ``backbone'' subgraph guarantees that property \textit{(b)} (and hence also \textit{(a)}) holds.

In Section \ref{section:concluding} we conclude the paper by formulating several conjectures regarding globally linked pairs. We also show, by giving an example, that 
being globally linked in $G$ in $\RR^3$ is not, in general, a generic property, even if $G$ has a connected rigidity matroid.
Finally, we discuss some algorithmic questions related to \mainclass[d] graphs.

\section{Preliminaries}\label{section:preliminaries}

For a graph $G = (V,E)$ and a set of vertices $X \subseteq V$, we let $G[X]$ denote the subgraph of $G$ induced by $X$. We follow the standard terminology and say that $X$ is a \emph{clique} in $G$ if $G[X]$ is a complete graph. We use $N_G(X)$ to denote the set of neighbors of $X \subseteq V$ in $G$. Occasionally we shall use the notation $V(G)$ and $E(G)$ to denote the vertex and edge sets, respectively, of a graph $G$.
We shall often write enumerations of the form ``for each $i \in \{k,\ldots,l\}$'', where potentially $l < k$. In this case $\{k,\ldots,l\}$ is understood to be the empty set and the statement in question is vacuous.


\subsection{Separators and fragments}\label{subsection:fragments}

Let $G = (V,E)$ be a connected graph. A set $S \subseteq V$ is a \emph{separator} of $G$ if $G-S$ is disconnected. If $|S| = d$, then $S$ is a \emph{$d$-separator}. 
We say that $S$ \emph{separates} the vertices $u,v \in V - S$, or that $S$ is \emph{$(u,v)$-separating}, if $u$ and $v$ lie in different components of $G-S$. 
The graph $G$ is \emph{$d$-connected} if it has at least $d+1$ vertices and it has no separators of size at most $d-1$.
The \emph{vertex connectivity} of $G$, denoted by $\kappa(G)$, is the largest $d$ for which $G$ is $d$-connected. For a pair of vertices $u,v \in V$, we use $\kappa(G,u,v)$ to denote the maximum number of internally vertex-disjoint paths from $u$ to $v$ in $G$. If $u$ and $v$ are nonadjacent in $G$, then by Menger's theorem this coincides with the minimum size of a $(u,v)$-separating vertex set in $G$. 
A \emph{minimum vertex cut} of $G$ is a separator of size $\kappa(G)$. Note that if $S$ is a minimum vertex cut, then every vertex in $S$ has a neighbor in every component of $G - S$. 

Throughout the rest of this subsection let $d \geq 1$ be an integer and let $G$ be a $d$-connected graph. Let $S,T \subseteq V$ be $d$-separators of $G$. We say that \emph{$S$ crosses $T$} if $S$ contains vertices from at least two components of $G-T$. The proof of the following (folklore) lemma, as well as the proofs of the subsequent lemmas in this subsection, are given in \cref{appendix:fragments}.
\begin{lemma}\label{lemma:crossingcuts}
Let $G = (V,E)$ be a $d$-connected graph and let $S,T \subseteq V$ be $d$-separators of $G$. Then the following are equivalent.
\begin{enumerate}[label=\textit{\alph*)}]
    \item $S$ crosses $T$.
    \item $T$ contains a vertex from each component of $G-S$.
    \item $T$ crosses $S$.
    \item $S$ contains a vertex from each component of $G-T$.
\end{enumerate}
\end{lemma}

By \cref{lemma:crossingcuts}, the relation of ``being crossing'' is symmetric. Thus we say that the $d$-separators $S$ and $T$ are \emph{crossing} if $S$ crosses $T$, or equivalently, if $T$ crosses $S$. Otherwise, we say that $S$ and $T$ are \emph{noncrossing}. This means that there is a single component $C$ of $G-S$ such that $T-S \subseteq C$, or equivalently, by \cref{lemma:crossingcuts}, there is a single component $C'$ of $G-T$ such that $S - T \subseteq C'$.

Let $u,v \in V$ be an ordered pair of vertices and let $S$ and $T$ be $(u,v)$-separating $d$-separators. We write $S \preceq_{u,v} T$ if $T-S$ and $v$ lie in the same component of $G-S$. In this case $S$ and $T$ are noncrossing, and by the symmetry of being noncrossing this is equivalent to the condition that $S-T$ and $u$ lie in the same component of $G-T$. 
The following lemma can be deduced from results in \cite{escalante_1972}.
In \cref{appendix:fragments}, we provide a proof which uses our terminology.

\begin{lemma}\label{minimumvertexcutorder}
Let $G$ be a $d$-connected graph. The relation $\preceq_{u,v}$ defines a partial order on the set of $(u,v)$-separating $d$-separators of $G$. Two $(u,v)$-separating $d$-separators are comparable in this order if and only if they are noncrossing.
\end{lemma}

We shall use the following observation about chains of $(u,v)$-separating $d$-separators.
\begin{lemma}\label{lemma:chainofvertexcuts}
Let $G$ be a $d$-connected graph. Let $k \geq 2$ be an integer and let $S_1 \succeq_{u,v} S_2 \succeq_{u,v} \ldots \succeq_{u,v} S_k$ be a chain of $d$-separators of $G$. If $S_{k-1} \neq S_k$, then $S_k - \cup_{i = 1}^{k-1} S_i$ is nonempty.
\end{lemma}

A \emph{fragment} of $G$ is a nonempty set of vertices $X \subseteq V$ such that $V-X-N_G(X)$ is nonempty. This means that $N_G(X)$ is a separator and $X$ is the union of some, but not all components of $G-N_G(X)$. The \emph{complement} of $X$, which we denote by $\overline{X}$, is $V - X - N_G(X)$. If $|N_G(X)| = d$, then we say that $X$ is a \emph{$d$-fragment}. For a pair of vertices $u$ and $v$, a fragment $X$ is said to be \emph{$(u,v)$-separating} if $u \in X$ and $v \in \overline{X}$ or $v \in X$ and $u \in \overline{X}$. 
We say that the $d$-fragments $X$ and $X'$ are \emph{crossing} (resp.\ \emph{noncrossing}) if $N_G(X)$ and $N_G(X')$ are crossing (resp.\ noncrossing) $d$-separators. 
We shall need the following result about noncrossing $d$-fragments.
\begin{lemma}\label{lemma:uxseparator}
    Let $G = (V,E)$ be a $d$-connected graph and $u,v,x \in V$, and let $X$ and $X'$ be noncrossing $d$-fragments with $x \in N_G(X)$. If $X$ is $(u,v)$-separating and $X'$ is $(u,x)$-separating, then $X'$ is also $(u,v)$-separating.
\end{lemma}

We say that a subset $V' \subseteq V$ of vertices is a \emph{$d$-block} of $G$ if for all $u,v \in V'$, either $uv \in E$ or $\kappa(G,u,v) \geq d$, and $V'$ is inclusion-wise maximal with respect to this property.

\begin{lemma}\label{lemma:dblock}
Let $G = (V,E)$ be a $d$-connected graph and let $X \subseteq V$ be a $d$-fragment of $G$. Suppose that $N_G(X)$ is a clique in $G$ and let $G'$ denote the subgraph of $G$ induced by $X \cup N_G(X)$. Then
\begin{enumerate}[label=\textit{\alph*)}]
    \item $\kappa(G',u,v) = \kappa(G,u,v)$ for all $u \in X$ and $v \in X \cup N_G(X)$.
    \item The $d$-separators of $G'$ are precisely the $d$-separators of $G$ that intersect $X$, as well as possibly $N_G(X)$.
    \item The $(d+1)$-blocks of $G'$ are precisely the $(d+1)$-blocks of $G$ that intersect $X$, as well as possibly $N_G(X)$.
\end{enumerate}
\end{lemma}

\begin{lemma}\label{lemma:inclusionwiseminimalfragment}
Let $G = (V,E)$ be a $d$-connected graph and let $X$ be an inclusion-wise minimal $d$-fragment of $G$. If there are no crossing $d$-separators in $G$, then $X \cup N_G(X)$ is a $(d+1)$-block of $G$.
\end{lemma}

\subsection{Rational functions and rational maps}

We shall need some definitions from algebraic geometry, which we introduce adapted to the affine space $\RR^n$. For a thorough introduction to the notion of rational maps, see, e.g., \cite[Chapter 5]{cox.etal_2015}. 

A \emph{rational function} (over $\RR$, in the variables $x_1,\ldots,x_n$) is an element of the field of fractions of the polynomial ring $\RR[x_1,\ldots,x_n]$. A rational function $\varphi$ can be represented as a quotient $f/g$ of polynomials $f,g \in \RR[x_1,\ldots,x_n]$, where $g$ is not the zero polynomial. Two pairs $f_1/g_1$ and $f_2/g_2$ represent the same rational function if and only if $f_1g_2 - f_2g_1$ is the zero polynomial. The \emph{domain of definition} of a rational function $\varphi$ is the set \[\dom(\varphi) = \{x \in \RR^n: \varphi \text{ has a representative } f/g \text{ with } g(x) \neq 0\}.\] The \emph{value} $\varphi(x)$ of $\varphi$ at a point $x \in \dom(\varphi)$ is $f(x)/g(x)$, where $f/g$ is a representative of $\varphi$ with $g(x) \neq 0$. It is easy to see that this is well-defined.

A \emph{rational map} from $\RR^n$ to $\RR^m$, denoted by $F : \RR^n \dashrightarrow \RR^m$, is an $m$-tuple of rational functions $(\varphi_1,\ldots,\varphi_m)$ in $n$ variables. The \emph{domain of definition} of $F$ is $\dom(F) = \cap_{i=1}^m \dom(\varphi_i)$; this is a dense open subset of $\RR^n$. The \emph{value} of $F$ at a point $x \in \dom(F)$ is $F(x) = (\varphi_1(x),\ldots,\varphi_m(x)).$ Thus, we may view $F$ as a function from $\dom(F)$ to $\RR^m$. 
We say that $F$ has \emph{rational coefficients} if each $\varphi_i$ has a representative $f/g$ with $f$ and $g$ having rational coefficients.

We say that a rational map $F : \RR^n \dashrightarrow \RR^m$ is \emph{dominant} if there is no nonzero polynomial $f \in \RR[y_1,\ldots,y_m]$ such that $F(\dom(F)) \subseteq \{y \in \RR^m: f(y) = 0\}$. In particular, this holds if $F(\dom(F))$ is a dense subset of $\RR^m$.
If $F: \RR^n \dashrightarrow \RR^m$ and $G: \RR^m \dashrightarrow \RR^k$ are dominant rational maps with $F = (\varphi_1,\ldots,\varphi_m)$ and $G = (\psi_1,\ldots,\psi_k)$, then we can define their composition $G \circ F$ in the natural way by taking representatives $\varphi_i = f_i / g_i, i \in \{1,\ldots,m\}$ and $\psi_j = f'_j/g'_j, j \in \{1,\ldots,k\}$ and letting the $j$-th component of $G \circ F$ be represented by 
\[\Psi_j = \frac{f'_j(f_1/g_1,\ldots,f_m/g_m)}{g'_j(f_1/g_1,\ldots,f_m/g_m)},\]which is a rational function in the variables $x_1,\ldots,x_n$. This is well-defined, and the condition that $F$ is dominant ensures that the denominator in the above expression is not identically zero. The rational map $G \circ F: \RR^n \dashrightarrow \RR^k$ defined by $(\Psi_1,\ldots,\Psi_k)$ is again dominant, and if $F$ and $G$ have rational coefficients, then so does $G \circ F$. The domain of definition of $G \circ F$ includes (but may be larger than) $\{x \in \RR^n: x \in \dom(F) \text{ and } F(x) \in \dom(G)\}$, and at a member $x$ of the latter set we have $(G\circ F)(x) = G(F(x))$.

Let us say that a point $x \in \RR^n$ is \emph{generic} if its $n$ coordinates are algebraically independent over $\QQ$; that is, if there is no nonzero polynomial $f \in \QQ[x_1,\ldots,x_n]$ such that $f(x) = 0$. The following lemma captures the easy but useful observation that a rational map with rational coefficients is completely determined by its value at any generic point.

\begin{lemma}\label{lemma:rationalmapsgeneric}
Let $F,G: \RR^n \dashrightarrow \RR^m$ be rational maps with rational coefficients and let $x \in \RR^n$ be a generic point. Then we have $x \in \dom(F) \cap \dom(G)$. Furthermore, if $F(x) = G(x)$, then $F$ and $G$ are equal as rational maps.
\end{lemma}
\begin{proof}
    Let $(f_1/g_1,\ldots,f_m/g_m)$ and $(f'_1/g'_1,\ldots,f'_m/g'_m)$ be representations of $F$ and $G$, respectively, with rational coefficients. Since $x$ is generic, we have that $g_i(x) \neq 0$ and $g'_i(x) \neq 0$ for every $i \in \{1,\ldots,m\}$. It follows that $x \in \dom(F) \cap \dom(G)$. If $F(x) = G(x)$, then for each $i \in \{1,\ldots,m\}$ we have
    \[\frac{f_i(x)}{g_i(x)} = \frac{f'_i(x)}{g'_i(x)}.\]
    By rearranging we get $f_i(x)g'_i(x) - f'_i(x)g_i(x) = 0$. Since $f_ig'_i - f'_ig_i$ is a polynomial with rational coefficients and $x$ is generic, this means that $f_ig'_i - f'_ig_i$ is the zero polynomial, so $f_i/g_i$ and $f'_i/g'_i$ represent the same rational function. Since this holds for each index $i$, we have $F = G$, as claimed. 
\end{proof}

\subsection{Congruences and reflections}

A function $f: \RR^d \rightarrow \RR^d$ is a \emph{congruence} if it preserves distances, that is, if $\norm{x - y} = \norm{f(x) - f(y)}$ for every $x, y \in \RR^d$, where $\norm{\cdot}$ denotes the Euclidean norm. It is well-known that every congruence of $\RR^d$ has the form $x \mapsto Qx + b$, where $Q \in \RR^{d \times d}$ is an orthogonal matrix and $b \in \RR^d$ is a vector.
We shall use the following standard facts.

\begin{lemma}\cite[Theorem 4]{rees_1983}\label{lemma:congruence}
Let $\alpha : \RR^d \rightarrow \RR^d$ be a congruence. If $\alpha$ fixes $d$ affinely independent points $p_1,\ldots,p_d \in \RR^d$, then $\alpha$ is either the identity map, or the orthogonal reflection in the affine hyperplane spanned by $p_1,\ldots,p_d$. 
\end{lemma}

\begin{lemma}\cite[(the proof of) Theorem 2]{rees_1983}\label{lemma:trilateration}
Let $S \subseteq \RR^d$ be a set of $d+1$ affinely independent points and let $p,p' \in \RR^d$ be a pair of points. If $\norm{p - s} = \norm{p' - s}$ for all $s \in S$, then $p = p'$.
\end{lemma}

Let $X \subseteq (\RR^d)^{d+1}$ denote the set of vector tuples $\left(a_1,\ldots,a_d,x\right)$ such that $\{a_1,\ldots,a_d\}$ is affinely independent, and let $F : X \to \RR^d$ be the map that sends each tuple  $\left(a_1,\ldots,a_d,x\right)$ to the orthogonal reflection of $x$ in the affine hyperplane spanned by $\{a_1,\ldots,a_d\}$. This turns out to be a rational map with rational coefficients, a fact that we shall repeatedly exploit in \cref{section:partialreflections}.

\begin{lemma}\label{lemma:rationalmap}
The map $F$ defined above is (the function induced by) a rational map with rational coefficients.
\end{lemma}
\begin{proof}
    It is well-known that given a matrix $M \in \RR^{d \times (d-1)}$ of full column rank, the orthogonal projection onto the linear subspace of $\RR^d$ spanned by the columns of $M$ is given by $M(M^TM)^{-1}M^T$, see \cite[Section 5.13]{meyer_2000}. From this it is not difficult to derive that the orthogonal reflection of $x$ in the affine hyperplane generated by $\{a_1,\ldots,a_d\}$ is given by
    \[(I-2M(M^TM)^{-1}M^T)(x - a_1) + a_1,\]
    where $M$ is the matrix having $a_2-a_1,\ldots,a_d-a_1$ as its columns. Note that $\{a_1,\ldots,a_d\}$ being affinely independent implies that $M^TM$ is invertible.
    It is well-known (and follows from, e.g., Cramer's rule) that the map $X \mapsto X^{-1}$ corresponds to a rational map $\RR^{d \times d} \dashrightarrow \RR^{d \times d}$ with rational coefficients whose domain is the set of invertible matrices. It follows that the above expression for orthogonal reflections corresponds to a rational map with rational coefficients whose domain includes the set of vector tuples $\left(a_1,\ldots,a_d,x\right)$ with $\{a_1,\ldots,a_d\}$ affinely independent. 
\end{proof}

\subsection{Rigid and globally rigid graphs}

Let $G = (V,E)$ be a graph and let $d \geq 1$ be an integer. As defined in the introduction, a \emph{realization} of $G$ in $\RR^d$ is a pair $(G,p)$, where $p$ is a map from $V$ to $\RR^d$. We also say that the pair $(G,p)$ is a \emph{$d$-dimensional framework}. The \emph{length} of an edge $uv \in E$ in $(G,p)$ is $\norm{p(u) - p(v)}$.
Two frameworks $(G,p)$ and $(G,q)$ in $\RR^d$ are \emph{equivalent} if for every edge $uv \in E$, the length of $uv$ is the same in $(G,p)$ and $(G,q)$. The frameworks are \emph{congruent} if there is a congruence $\RR^d \rightarrow \RR^d$ that maps $p(v)$ to $q(v)$ for all $v \in V$. This is well-known to be equivalent to the condition that $\norm{p(u) - p(v)} = \norm{q(u) - q(v)}$ holds for every pair of vertices $u,v \in V$.

A $d$-dimensional framework $(G,p)$ is \emph{rigid} if there is some $\varepsilon > 0$ such that every equivalent framework $(G,q)$ with $\norm{p(v) - q(v)} < \varepsilon$ for all $v \in V$ is congruent to $(G,p)$. If this holds globally, that is, if every equivalent framework is congruent to $(G,p)$, then we say that $(G,p)$ is \emph{globally rigid}.

It turns out that the combinatorial structure of $G$ determines the (global) rigidity of $(G,p)$ for ``almost all'' $p$. To be more precise, we say that a $d$-dimensional framework $(G,p)$ is \emph{generic} if the multiset of the $d \cdot |V|$ coordinates of $\{p(v), v\in V\}$ is algebraically independent over $\QQ$; in other words, if $p$ is generic as a point in $\RR^{d|V|}$. It is known that if some generic $d$-dimensional realization of a graph is rigid (resp.\ globally rigid), then every generic $d$-dimensional realization of the graph is rigid (resp.\ globally rigid), see \cite{asimow.roth_1978,connelly_2005,gortler.etal_2010}.
Thus, we say that a graph $G$ is \emph{rigid in $\RR^d$} if its generic realizations in $\RR^d$ are rigid, and \emph{globally rigid in $\RR^d$} if its generic realizations in $\RR^d$ are globally rigid. 


It is well-known that if a graph on at least $d+1$ vertices is rigid in $\RR^d$, then it is $d$-connected. The next lemma describes some basic gluing results regarding graph rigidity. The first part can be found in, e.g., \cite[Lemma 11.1.9(a)]{whlong}, while the second is easy to deduce from the definitions.

\begin{lemma}\label{lemma:rigidgluing}
Let $G = (V,E)$ be the union of the graphs $G_1 = (V_1,E_1), G_2 = (V_2,E_2)$. If $G_1$ and $G_2$ are rigid in $\RR^d$ and $|V_1 \cap V_2| \geq d$, then $G$ is rigid in $\RR^d$. Conversely, if $G$ is rigid in $\RR^d$ and $V_1 \cap V_2$ is a clique in both $G_1$ and $G_2$, then $G_1$ and $G_2$ are rigid in $\RR^d$.
\end{lemma}

Given a graph $G$, its \emph{cone graph} is obtained from $G$ by adding a new
vertex $v$ and new edges from $v$ to each vertex of $G$.

\begin{theorem}\cite{whiteley_cones}\label{theorem:coning}
A graph is rigid in $\RR^d$ if and only if its cone graph is rigid in $\RR^{d+1}$. 
\end{theorem}

\subsection{The rigidity matrix and the rigidity matroid}

The rigidity matroid of a graph $G$ is a matroid defined on the edge set
of $G$ which reflects the rigidity properties of all generic realizations of
$G$. 
Let $(G,p)$ be a realization of a graph $G=(V,E)$ in $\RR^d$.
The {\it rigidity matrix} of the framework $(G,p)$
is the matrix $R(G,p)$ of size
$|E|\times d|V|$, where, for each edge $v_iv_j\in E$, in the row
corresponding to $v_iv_j$,
the entries in the $d$ columns corresponding to vertices $v_i$ and $v_j$ contain
the $d$ coordinates of
$(p(v_i)-p(v_j))$ and $(p(v_j)-p(v_i))$, respectively,
and the remaining entries
are zeros.
The rigidity matrix of $(G,p)$ defines
the {\it rigidity matroid}  of $(G,p)$ on the ground set $E$
by linear independence of rows. It is known that any pair of generic frameworks
$(G,p)$ and $(G,q)$ have the same rigidity matroid.
We call this the $d$-dimensional {\it rigidity matroid}
${\cal R}_d(G)=(E,r_d)$ of the graph $G$.

We denote the rank of ${\cal R}_d(G)$ by $r_d(G)$.
A graph $G=(V,E)$ is {\it $\Rd$-independent} if $r_d(G)=|E|$, and it is an {\it $\Rd$-circuit}  if it is not $\Rd$-independent, but every proper 
subgraph of $G$ is $\Rd$-independent. We note that in the literature such graphs are sometimes called $M$-independent in $\RR^d$ and $M$-circuits in $\RR^d$, respectively. 
A pair $\{u,v\}$ of nonadjacent vertices in $G$ is \emph{linked in $G$ in $\RR^d$} if $r_d(G) = r_d(G+uv)$. Equivalently, $\{u,v\}$ is linked in $G$ in $\RR^d$ if there is an $\Rd$-circuit in $G+uv$ that contains $uv$.

We shall use the following characterization of rigid graphs. 

\begin{theorem}\label{theorem:gluck}\cite{gluck}
Let $G = (V,E)$ be a graph on at least $d$ vertices. Then $G$ is rigid in $\RR^d$ if and only if $r_d(G) = d|V| - \binom{d+1}{2}$. 
\end{theorem}

Let ${\cal M}$ be a matroid on ground set $E$ with rank function $r$. 
We can define a relation on the pairs of elements of $E$ by
saying that $e,f\in E$ are
equivalent if $e=f$ or there is a circuit $C$ of ${\cal M}$
with $\{e,f\}\subseteq C$.
This defines an equivalence relation. The equivalence classes are 
the {\it connected components} of ${\cal M}$.
The matroid is said to be {\it connected} if there is only one equivalence class.
Given a graph $G=(V,E)$, the subgraphs induced by the edge sets of the connected components
of ${\cal R}_d(G)$ are the
{\it $\Rd$-connected components} of $G$.
The graph is said to be {\it $\Rd$-connected} if ${\cal R}_d(G)$ is connected.

\subsection{Globally linked pairs}

Let $(G,p)$ be a framework in $\RR^d$. As a ``granular'' counterpart to global rigidity, we define a pair of vertices $\{u,v\}$ to be \emph{globally linked} in $(G,p)$ if for every equivalent framework $(G,q)$ in $\RR^d$, we have $\norm{p(u) - p(v)} = \norm{q(u) - q(v)}$. Thus, $(G,p)$ is globally rigid if and only if every pair of vertices is globally linked in $(G,p)$. In contrast with global rigidity, this is not, in general, a generic property: it may happen that (in a given dimension) $\{u,v\}$ is globally linked in some generic realizations and not globally linked in others, see \cite[Figs. 1 and 2]{jackson.etal_2006}. Nevertheless, we define the pair $\{u,v\}$ to be \emph{globally linked in $G$ in $\RR^d$} if $\{u,v\}$ is globally linked in every generic realization of $G$ in $\RR^d$.


We shall need the following ``trilateration'' result concerning globally linked pairs.

\begin{lemma}\label{lemma:globallylinkedtransitive}
Let $G = (V,E)$ be a graph and let $S \subseteq V$ be a clique of size $d+1$ in $G$. Let $u$ and $v$ be a pair of vertices in $G$. If $\{u,s\}$ and $\{v,s\}$ are globally linked in $G$ in $\RR^d$ for every $s \in S$, then $\{u,v\}$ is globally linked in $G$ in $\RR^d$.  
\end{lemma}
\begin{proof}
Let $(G,p)$ be a generic $d$-dimensional realization of $G$ and $(G,q)$ an equivalent realization. Since $G[S]$ is complete, it is globally rigid in $\RR^d$, so there is a congruence that maps $(G[S],q)$ to $(G[S],p)$. By applying this congruence to all of $(G,q)$, we may suppose that $p(s) = q(s)$ for every $s \in S$.

Using the assumption that $\{u,s\}$ is globally linked in $G$ in $\RR^d$, we have $\norm{p(u) - p(s)} = \norm{q(u) - q(s)} = \norm{q(u) - p(s)}$ for every $s \in S$. Since $(G,p)$ is generic, $\{p(s), s \in S\}$ is affinely independent, so \cref{lemma:trilateration} implies $p(u) = q(u)$. By the same reasoning we also have $p(v) = q(v)$, so in particular $\norm{p(u) - p(v)} = \norm{q(u) - q(v)},$ as desired.
\end{proof}

Let $G$ be a graph. The \emph{globally linked closure of $G$ in $\RR^d$}, denoted by $\glclosure(G)$, is the graph obtained from $G$ by adding the edge $uv$ for every pair $\{u,v\}$ that is globally linked in $G$ in $\RR^d$. The following lemma, which is easy to deduce from the definitions, shows that $\glclosure$ is indeed a closure operator.
\begin{lemma} We have the following for all $d \geq 1$ and all graphs $G$.
\begin{itemize}
    \item $E(G) \subseteq E(\glclosure(G))$,
    \item $\glclosure(\glclosure(G)) = \glclosure(G)$, and
    \item for all graphs $G'$ with $E(G') \subseteq E(G)$, $E(\glclosure(G')) \subseteq E(\glclosure(G))$.
\end{itemize} 
\end{lemma}
The following observation captures a key property of taking the globally linked closure of a graph.

\begin{lemma}\label{lemma:equivclasses}
    Let $G = (V,E)$ be a graph and let $(G,p)$ and $(G,q)$ be $d$-dimensional realizations of $G$ with $(G,p)$ generic. Then $(G,p)$ and $(G,q)$ are equivalent if and only if $(\glclosure(G),p)$ and $(\glclosure(G),q)$ are equivalent. Consequently, $G$ is rigid (resp.\ globally rigid) in $\RR^d$ if and only if $\glclosure(G)$ is rigid (resp.\ globally rigid) in $\RR^d$.
\end{lemma}
\begin{proof}
    Clearly if $(\glclosure(G),p)$ and $(\glclosure(G),q)$ are equivalent, then so are $(G,p)$ and $(G,q)$. On the other hand, if $(G,p)$ and $(G,q)$ are equivalent, then by the definition of being globally linked we have that $\norm{p(u) - p(v)} = \norm{q(u) - q(v)}$ for any pair $u,v \in V$ for which $\{u,v\}$ is globally linked in $G$ in $\RR^d$. Since $uv$ is an edge of $\glclosure(G)$ if and only if $\{u,v\}$ is globally linked in $G$ in $\RR^d$, it follows that $(\glclosure(G),p)$ and $(\glclosure(G),q)$ are equivalent. 
    
    From this and the relevant definitions it is immediate that $(G,p)$ is rigid (resp.\ globally rigid) if and only if $(\glclosure(G),p)$ is rigid (resp.\ globally rigid). Since $(G,p)$ is generic, this means that $G$ is rigid (resp.\ globally rigid) in $\RR^d$ if and only if $\glclosure(G)$ is rigid (resp.\ globally rigid) in $\RR^d$.
\end{proof}

Following \cite{jackson.etal_2006}, we define a \emph{$d$-dimensional globally rigid cluster} 
of $G$ to be a maximal vertex set of $G$ in which each vertex pair is globally linked in $G$ in $\RR^d$. In other words, the $d$-dimensional globally rigid clusters of $G$ are the maximal cliques of $\glclosure(G)$.

The following simple observation is central to the paper.
\begin{lemma}\label{lemma:globallylinkedclosure}
Let $G = (V,E)$ be a graph and let $S \subseteq V$ be a set of at most $d$ vertices. If $S$ is a separator in $G$, then it is also a separator in $\glclosure(G)$, and moreover if $S$ is $(u,v)$-separating in $G$, then it is also $(u,v)$-separating in $\glclosure(G)$, for any pair of vertices $u,v \in V$. In particular, if $u,v \in V$ are nonadjacent and $\kappa(G,u,v) \leq d$, then $\{u,v\}$ is not globally linked in $G$ in $\RR^d$.
\end{lemma}
\begin{proof}
Let $u,v \in V$ be a pair of vertices such that $S$ is $(u,v)$-separating in $G$, and let $X$ be a $(u,v)$-separating fragment such that $N_G(X) = S$. Let $(G,p)$ be a generic realization of $G$ in $\RR^d$. 
For any generic realization $(G,p)$, we can reflect the points $p(x),x \in \overline{X}$ in a hyperplane $H$ containing $\{p(s), s \in S\}$ and leave the rest of the points in place to obtain an equivalent noncongruent realization $(G,q)$. Since $(G,p)$ is generic, we can choose $H$ in such a way that it avoids the points $\{p(v), v \in V - S\}$. With this choice we have $\norm{p(u') - p(v')} \neq \norm{q(u') - q(v')}$ for all pairs $u' \in X$ and $v' \in \overline{X}$. This shows that all such pairs $\{u',v'\}$ are not globally linked in $G$ in $\RR^d$, and thus there are no edges in $\glclosure(G)$ going between $X$ and $\overline{X}$. It follows that $S$ is still $(u,v)$-separating in $\glclosure(G)$, as required.
\end{proof}

\begin{corollary}\label{corollary:globallylinkedclosure}
    Let $G = (V,E)$ be a graph. The $d$-separators, the $d$-fragments and the $(d+1)$-blocks of $G$ and $\glclosure(G)$ coincide.
\end{corollary}

In the preceding proof if $|S| = d$, then the hyperplane along which we reflect is just the affine span of $\{p(x), x \in S\}$. Intuitively, $(G,q)$ is a ``partial reflection'' of $(G,p)$ through this hyperplane. Our first goal is to gain a deeper understanding of the operation of taking such partial reflections.





\section{Sequences of partial reflections}\label{section:partialreflections}

In this section we carefully define and analyze the operation of taking partial reflections. Although the notion of (sequences of) partial reflections is fairly intuitive, their analysis turns out to be rather technical and notation-heavy. Apart from the definition of a reduced sequence of partial reflections, which is central to our investigations, we shall only use the results of this section at a few key points in the next section. Nonetheless, we believe that this analysis may also be useful in the context of other problems involving the global rigidity of frameworks.

\subsection{Basic definitions}

Let $G = (V,E)$ be a graph, $X$ a $d$-fragment of $G$ and $S = N_G(X)$. We say that a $d$-dimensional framework $(G,p)$ is \emph{$\sep{X}$-admissible} if $p(S) = \{p(v) : v \in S\} \subseteq \RR^d$ is affinely independent. Given an $\sep{X}$-admissible framework $(G,p)$, we define the \emph{partial reflection of $(G,p)$ corresponding to $X$} to be the framework $(G,\sep{X}(p))$ given by

\begin{equation*}
\sep{X}(p)(v) = \begin{cases}
p(v) & v \in X \cup N_G(X),\\
\reflection{p(S)}(p(v)) & v \in \overline{X},
\end{cases}
\end{equation*}
where $\reflection{p(S)}$ denotes the orthogonal reflection in the affine hyperplane spanned by $p(S)$.
The following two lemmas record basic properties of partial reflections which we shall use without reference throughout the paper.

\begin{lemma}
Let $X$ be a $d$-fragment of $G$ and let $(G,p)$ be an $\sep{X}$-admissible framework in $\RR^d$. Then $(G,p)$ is equivalent to $(G,\sep{X}(p))$. Moreover, if there are vertices $u \in X$ and $v \in \overline{X}$ such that $p(u)$ and $p(v)$ are not in the affine hyperplane spanned by $p(S)$, then $(G,p)$ and $(G,\sep{X}(p))$ are not congruent. 
\end{lemma}

\begin{lemma}\label{lemma:congruentpartialreflection}
Let $X$ be a $d$-fragment of $G$. Consider an $\sep{X}$-admissible framework $(G,p)$ in $\RR^d$ and let $(G,q)$ be a congruent framework. Then $(G,q)$ is also $\sep{X}$-admissible and $(G,\sep{X}(p))$ and $(G,\sep{X}(q))$ are congruent.
\end{lemma}

It follows from \cref{lemma:rationalmap} that the mapping $p \mapsto \sep{X}(p)$ corresponds to a rational map $\RR^{d|V|} \dashrightarrow \RR^{d|V|}$ with rational coefficients whose domain of definition is the set \[\dom(\sep{X}) = \{p \in \RR^{d|V|}: (G,p) \text{ is } \sep{X}\text{-admissible}\}.\] We call this rational map the \emph{partial $d$-reflection corresponding to $X$}, and by a slight abuse of notation we also denote it by $\sep{X}$. This map is dominant: indeed, its image is also $\dom(\sep{X})$, a dense open subset of $\RR^{d|V|}$.
For convenience, we also define $\sep{X}$ to be the identity map of $\RR^{d|V|}$ whenever $X \subseteq V$ is not a $d$-fragment of $G$. In particular, we shall use this notation when $X = \varnothing$ and when $X = V-S$ for some $d$-separator $S$ of $G$.

A \emph{sequence of partial $d$-reflections} of $G$ is a tuple $F = (\sep{X_1},\ldots,\sep{X_k})$, where $X_1,\ldots,X_k$ are $d$-fragments of $G$. 
A framework $(G,p)$ is \emph{$F$-admissible} if $(G,p)$ is $X_1$-admissible and $(G,(\sep{X_{i-1}} \circ \ldots \circ \sep{X_1})(p))$ is $\sep{X_i}$-admissible for each $i \in \{2,\ldots,k\}$. 
For an $F$-admissible framework $(G,p)$ we define $F(p) = \sep{X_k} \circ \ldots \circ \sep{X_1}(p)$. In this way, we associate to each sequence of partial $d$-reflections $F$ a dominant rational map with rational coefficients whose domain of definition includes (but may be larger than) the set of $F$-admissible frameworks. Again, for convenience we shall use a slight abuse of notation and say that $F$ itself is a rational map. 

We also consider the empty sequence to be a sequence of partial $d$-reflections, consisting of zero reflections, and corresponding to the identity map on $\RR^{d|V|}$. We refer to this as the \emph{trivial} sequence of partial $d$-reflections. When defining a sequence of partial $d$-reflections $F = (\sep{X_1},\ldots,\sep{X_k})$, we shall usually allow the possibility that $k=0$ with the understanding that $F$ is the trivial sequence of partial $d$-reflections in this case.  We also consider a single partial $d$-reflection $\sep{X}$ to be a sequence of partial reflections, and we shall continue to write $\sep{X}$ instead of $(\sep{X})$ in this case. 
For a pair of sequences of partial $d$-reflections $F = (\sep{X_1},\ldots,\sep{X_k}), F' = (\sep{X_{k+1}},\ldots,\sep{X_l})$, we let $F' \circ F$ denote the sequence $(\sep{X_1},\ldots,\sep{X_k},\sep{X_{k+1}},\ldots,\sep{X_l})$ of partial $d$-reflections.

A generic framework in $\RR^d$ on at least $d+1$ vertices is in general position, and in particular $\sep{X}$-admissible for any partial $d$-reflection $\sep{X}$. The following lemma shows that partial reflections preserve genericity.

\begin{lemma}
Let $X$ be a $d$-fragment of $G = (V,E)$ and let $(G,p)$ be an $\sep{X}$-admissible framework in $\RR^d$. Then $(G,p)$ is generic if and only if $(G,\sep{X}(p))$ is generic.
\end{lemma}
\begin{proof}
$(G,p)$ is generic if and only if the field $\QQ(p)$ generated by the coordinates of $p$ over $\QQ$ has transcendence degree $d|V|$ over $\QQ$. Since $\sep{X}$ is a rational map with rational coefficients, we can take representatives $f_1/g_1,\ldots,f_{d|V|}/g_{d|V|}$ for its coordinate functions, where each $f_i$ and $g_i$ is a polynomial in $d|V|$ variables and with rational coefficients. Note that since $p$ is generic, we have $g_i(p) \neq 0$ for $i = 1,\ldots,d|V|$. It follows that \[\sep{X}(p) = \left(\frac{f_1(p)}{g_1(p)},\ldots,\frac{f_{d|V|}(p)}{g_{d|V|}(p)}\right).\]Now $f_i(p)/g_i(p) \in \QQ(p)$ for $i \in \{1,\ldots,d|V|\}$, and consequently $\QQ(\sep{X}(p)) \subseteq \QQ(p)$. But  $\sep{X}(\sep{X}(p)) = p$, so by symmetry we also have $\QQ(\sep{X}(p)) \supseteq \QQ(p)$ and thus $\QQ(\sep{X}(p)) = \QQ(p)$. This shows that if one of the frameworks is generic, then so is the other.
\end{proof}

It follows that generic frameworks are $F$-admissible for every sequence of partial $d$-reflections $F$.

\subsection{Equivalent and reduced sequences}

From now on, we concentrate on the effect of (sequences of) partial $d$-reflections on generic frameworks. To this end, we introduce the following notion.
We say that two sequences of partial $d$-reflections $F_1$ and $F_2$ of $G$ are \emph{equivalent} (denoted by $F_1 \partialeq F_2$) if for every generic framework $(G,p)$ in $\RR^d$, the frameworks $(G,F_1(p))$ and $(G,F_2(p))$ are congruent. As the name suggests, this defines an equivalence relation on the set of sequences of partial $d$-reflections of $G$. The following lemma shows that equivalence of sequences of partial reflections is, in a sense, a ``generic'' property.

\begin{lemma}\label{lemma:equivgeneric}
Let $G$ be a graph and let $F_1,F_2$ be sequences of partial $d$-reflections. The following are equivalent. 
\begin{enumerate}
    \item $(G,F_1(p))$ and $(G,F_2(p))$ are congruent for some generic $d$-dimensional realization $(G,p)$.
    \item $F_1 \partialeq F_2$.
    \item $(G,F_1(p))$ and $(G,F_2(p))$ are congruent for every $d$-dimensional framework $(G,p)$ that is both $F_1$-admissible and $F_2$-admissible. 
\end{enumerate}
\end{lemma}
\begin{proof}
The implications \textit{(c)} $\Rightarrow$ \textit{(b)} $\Rightarrow$ \textit{(a)} are trivial, so it suffices to show \textit{(a)} $\Rightarrow$ \textit{(c)}.
Let $u$ and $v$ be a pair of vertices of $G$ and consider the condition
\begin{equation} \label{eq:equivalentsequences}
\norm{F_1(p)(u) - F_1(p)(v)}^2 = \norm{F_2(p)(u) - F_2(p)(v)}^2    
\end{equation}
for a realization $(G,p)$. Both sides of this equation are rational maps with rational coefficients in the coordinates of $p$. By \cref{lemma:rationalmapsgeneric} if equality holds for some generic $p$, then the two sides are equal as rational maps, and in that case equality must hold for any framework that is in a domain of both; in particular, whenever $(G,p)$ is both $F_1$-admissible and $F_2$-admissible. Applying this argument for every pair $u,v$ finishes the proof. 
\end{proof}

The next lemma, which is an immediate consequence of \cref{lemma:congruentpartialreflection}, shows that equivalence behaves well under composition. We shall use this observation without reference in the following.

\begin{lemma}\label{lemma:compositionequivalence}
Let $G$ be a graph and let $F_1,F_2,F_1',F_2'$ be sequences of partial $d$-reflections of $G$ with $F_1 \partialeq F_1'$ and $F_2 \partialeq F_2'$. Then $F_2 \circ F_1 \partialeq F_2' \circ F_1'$.
\end{lemma}

The next three lemmas describe examples of equivalent sequences of partial $d$-reflections. They will be useful for ``simplifying'' sequences of partial $d$-reflections, up to equivalence.

\begin{lemma}\label{lemma:equivcomplement}
Let $G = (V,E)$ be a graph and let $X$ be a $d$-fragment of $G$. 
Then $\sep{X} \partialeq \sep{\overline{X}}$.
\end{lemma}
\begin{proof}
Let $S = N_G(X)$ be the separator corresponding to $X$. Then for any $\sep{X}$-admissible framework $(G,p)$, the orthogonal reflection $\reflection{p(S)}$ maps $(G,\sep{X}(p))$ to $(G,\sep{\overline{X}}(p))$, so the two frameworks are indeed congruent.
\end{proof}

Recall that a $d$-fragment of $G$ can be viewed as the union of some, but not all components of the graph obtained by deleting a $d$-separator from $G$. This shows that if $X$ and $X'$ are $d$-fragments with $N_G(X) = N_G(X')$, then $X \Delta X'$ is either a $d$-fragment, the empty set, all of $G - N_G(X)$. (Here $X \Delta X'$ denotes the symmetric difference of $X$ and $X'$.)
The next lemma is immediate from the definition of a partial reflection and our convention that $\sep{\varnothing}$ and $\sep{V - N_G(X)}$ both correspond to the identity map. 

\begin{lemma}\label{lemma:equivdelta}
Let $G = (V,E)$ be a graph and let $X,X'$ be a pair of $d$-fragments with $N_G(X) = N_G(X')$. 
Then $\sep{X} \circ \sep{X'} \partialeq \sep{X \Delta X'}$.
\end{lemma}

Finally, we show that partial reflections corresponding to noncrossing $d$-fragments commute with each other, up to equivalence. 

\begin{lemma}\label{lemma:equivnoncrossing}
Let $G = (V,E)$ be a $d$-connected graph and let $X,X'$ be a pair of noncrossing $d$-fragments. 
Then $\sep{X} \circ \sep{X'} \partialeq \sep{X'} \circ \sep{X}$.
\end{lemma}
\begin{proof}
Let $S = N_G(X)$ and $T = N_G(X')$ be the $d$-separators corresponding to $X$ and $X'$. By assumption, $S$ and $T$ are noncrossing, so there is a component $C$ of $G - S$ such that $T - S \subseteq C$, and a component $C'$ of $G - T$ such that $S - T \subseteq C'$. By possibly replacing $X$ with $\overline{X}$ and $X'$ with $\overline{X'}$ and using \cref{lemma:equivcomplement} we may suppose that $C \subseteq X$ and $C' \subseteq X'$. With this assumption, for any generic framework $(G,p)$ in $\RR^d$ and any vertex $v \in S \cup T$ we have $p(v) = \sep{X}(p)(v) = \sep{X'}(p)(v)$. In other words, both $\sep{X}$ and $\sep{X'}$ leave the vertices in $S \cup T$ in place. This implies $\sep{X} \circ \sep{X'}(p)(v) = \sep{X'} \circ \sep{X}(p)(v)$ for every vertex $v \in V$, so $\sep{X} \circ \sep{X'} \partialeq \sep{X'} \circ \sep{X}$ indeed holds.
\end{proof}

\cref{lemma:equivdelta,lemma:equivnoncrossing} allow us to define a ``canonical form'' for sequences of partial $d$-reflections, at least when the fragments involved are noncrossing.
Let $G$ be a $d$-connected graph. A sequence $F$ of partial $d$-reflections of $G$ is \emph{reduced} if either it is the trivial sequence, or $F = (\sep{X_1},\ldots,\sep{X_k})$ where $X_1,\ldots,X_k$ are pairwise noncrossing $d$-fragments and $N_G(X_1),\ldots,N_G(X_k)$ are pairwise different. 

\begin{corollary}\label{corollary:reducedsequence}
Let $G = (V,E)$ be a $d$-connected graph and suppose that $G$ has no crossing $d$-separators. Then every sequence of partial $d$-reflections of $G$ is equivalent to a reduced sequence of partial $d$-reflections of $G$.
\end{corollary}
\begin{proof}
This follows from repeated applications of \cref{lemma:equivdelta,lemma:equivnoncrossing}.
\end{proof}

We note that in the case of a nontrivial reduced sequence $F$, the $F$-admissible frameworks are precisely those frameworks $(G,p)$ in which $\{p(x), x \in N_G(X_i)\}$ is affinely independent for each $d$-fragment $X_i$ appearing in $F$. This follows from the observation that if $X_i$ and $X_j$ are noncrossing $d$-fragments, then $\sep{X_j}$ acts as a congruence (either as the identity or as a reflection) on $\{p(x), x \in N_G(X_i)\}$.

The following two results show that reduced sequences act, in a sense, as freely as possible on generic frameworks.

\begin{lemma}\label{lemma:mainsimple}
Let $G = (V,E)$ be a $d$-connected graph and let $F = (\sep{X_1},\ldots,\sep{X_k})$ be a nontrivial reduced sequence of partial $d$-reflections. Let $u,v \in V$ be a pair of vertices and suppose that there is an index $i \in \{1,\ldots,k\}$ for which $X_i$ is $(u,v)$-separating. Then for every generic realization $(G,p)$ in $\RR^d$ we have $\norm{p(u) - p(v)} \neq \norm{F(p)(u) - F(p)(v)}$.
\end{lemma}
\begin{proof}
We start by showing that we may assume that $X_1,\ldots,X_k$ are all $(u,v)$-separating in $G$. Indeed, assuming otherwise, by \cref{lemma:equivnoncrossing} we may reorder $\sep{X_1},\ldots,\sep{X_k}$ so that $X_1,\ldots,X_{k_0}$ are $(u,v)$-separating and $X_{k_0+1},\ldots,X_k$ are not $(u,v)$-separating, for some $k_0 \in \{1,\ldots,k-1\}$. Consider $F' = (\sep{X_{k_0+1}},\ldots,\sep{X_k})$. Since none of the fragments in $F'$ are $(u,v)$-separating, the distance between $u$ and $v$ is equal in $(G,q)$ and $(G,F'(q))$ for any $F'$-admissible framework $(G,q)$. It follows that we may disregard these partial reflections altogether and assume that $X_1,\ldots,X_k$ are all $(u,v)$-separating in $G$.

Let us consider the $d$-separators $S_i = N_{G}(X_i)$ for each  $i \in \{1,\ldots,k\}$.
We may assume, using \cref{minimumvertexcutorder} and \cref{lemma:equivnoncrossing} again, that  $S_k \preceq_{u,v} \ldots \preceq_{u,v} S_1$. Finally, by \cref{lemma:equivcomplement} we can assume that $u \in X_i$ for $i \in \{1,\ldots,k\}.$ Note that these assumptions ensure that for any pair of indices $1 \leq i < j \leq k$ we have $S_j \subseteq X_i \cup S_i$, so that  $\sep{X_i}$ fixes the image of $S_j$ in any $F$-admissible framework.

Suppose for a contradiction that 
\[\norm{p(u)-p(v)} = \norm{F(p)(u) - F(p)(v)}\] holds for some generic realization $(G,p)$. 
After squaring both sides, they become rational maps with rational coefficients in the coordinates of $p$. By \cref{lemma:rationalmapsgeneric}, if equality holds for the generic configuration $p$, then it holds for any configuration that is in the domain of both sides. This implies that for any $F$-admissible framework $(G,q)$ we have 
\[\norm{q(u)-q(v)} = \norm{F(q)(u) - F(q)(v)}.\]

\begin{figure}[t]
    \centering
    \begin{subfigure}[b]{0.49\linewidth}
    \centering
                \includegraphics[]{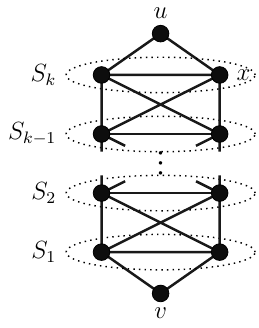}
        \caption{}
    \end{subfigure}
    \begin{subfigure}[b]{0.49\linewidth}
        \centering
                   \includegraphics[]{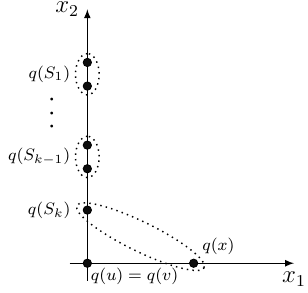}
        \caption{}
    \end{subfigure}
    \caption{The construction in the proof of \cref{lemma:mainsimple}. \textit{(a)} The graph $G$ and the $d$-separators $S_1,\ldots,S_k$. Here $d = 2$. \textit{(b)} The framework $(G,q)$ (the edges are not depicted) is $F$-admissible and is constructed in such a way that apart from $S_k$, the image of every $d$-separator corresponding to the fragments in $F$ lies in the $x_1 = 0$ hyperplane.}
    \label{fig:case1}
\end{figure}

We contradict this by giving an explicit counterexample. See \cref{fig:case1} for an example of the construction in the $d = 2$ case. From \cref{lemma:chainofvertexcuts} we have that $S_k - \cup_{i=1}^{k-1}S_i$ is nonempty, so let us fix a vertex $x$ from it. We define an $F$-admissible framework $(G,q)$ as follows. Let $q(u) = q(v) = (0,\ldots,0), q(x) = (1,0,\ldots,0),$ and let us choose the values $q(z), z \in V - \{u,v,x\}$ so that each point lies in the $x_1 = 0$ hyperplane, but the placement is relatively generic.\footnote{By being ``relatively generic'' we mean that the coordinates of $q$ that are unspecified form an algebraically independent set over $\QQ$.} Now $\{q(s), s \in S_i\}$ is contained in the $x_1 = 0$ hyperplane for each $i \in \{1, \ldots, k-1\}$, and thus by the relatively generic choice of coordinates its affine span is precisely the $x_1 = 0$ hyperplane.  On the other hand, $\{q(s), s \in S_k\}$ is not contained in this hyperplane, and by the choice of coordinates its affine span does not contain the origin. It follows that each partial reflection $\sep{X_1},\ldots,\sep{X_{k-1}}$ leaves $q(v)$ in place but $\sep{X_k}$ does not. Hence $F(q)(v) \neq (0,\ldots,0)$, while $F(q)(u) = (0,\ldots,0)$. This shows that \[\norm{F(q)(u) - F(q)(v)} \neq 0 = \norm{q(u) - q(v)},\]a contradiction. 
\end{proof}

We note that it can happen that $\norm{F(p)(u) - F(p)(v)} = \norm{p(u) - p(v)}$ if $(G,p)$ is a non-generic (but $F$-admissible) framework, see \cref{fig:nongeneric}.

\begin{figure}[t]
    \centering
        \includegraphics[]{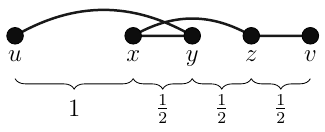}
        \caption{An example showing that the analogue of \cref{lemma:mainsimple} for non-generic frameworks is false. Here $d = 1$ and $G$ is the path of length four. Let $X_1 = \{u,x,y\}, X_2 = \{u,y\}$ and $X_3 = \{u\}$ be $1$-fragments of $G$ and let $F = (\sep{X_1},\sep{X_2}, \sep{X_3})$. The one-dimensional framework $(G,p)$ depicted is $F$-admissible and satisfies $\norm{F(p)(u) - F(p)(v)} = \norm{p(u) - p(v)}$.}
        \label{fig:nongeneric}
\end{figure}

\begin{corollary}\label{corollary:noncongruentreflections}
Let $G = (V,E)$ be a $d$-connected graph in which there are no crossing $d$-separators and let $F_1 = (\sep{X_1},\ldots,\sep{X_k})$ and $F_2 = (\sep{Y_1},\ldots,\sep{Y_l})$ be (possibly trivial) reduced sequences of partial $d$-reflections of $G$. Then $F_1 \partialeq F_2$ if and only if either both $F_1$ and $F_2$ are trivial, or $k = l$ and there is a permutation $\varphi$ of $\{1,\ldots,k\}$ such that $X_i = Y_{\varphi(i)}$ or $X_i = \overline{Y_{\varphi(i)}}$ for all $i \in \{1,\ldots,k\}.$ 
\end{corollary}
\begin{proof}
Sufficiency follows from \cref{lemma:equivcomplement,lemma:equivnoncrossing}. For necessity, note that $F_1 \partialeq F_2$ if and only if $F_1 \circ F_2 \partialeq \text{id}$, and that since there are no crossing $d$-separators in $G$, the fragments appearing in $F_1$ and $F_2$ are all noncrossing. We can obtain a reduced sequence $F$ equivalent to $F_1 \circ F_2$ by deleting $\sep{X_i}$ and $\sep{Y_j}$ and adding $\sep{X_i \Delta Y_j}$ for all pairs $X_i, Y_j$ such that $N_G(X_i) = N_G(Y_j)$. (That the sequence obtained in this way is equivalent to $F_1 \circ F_2$ follows from \cref{lemma:equivdelta,lemma:equivnoncrossing}.) The sequence $F$ obtained in this way is trivial if and only if it satisfies the condition in the statement. If it is nontrivial, then it contains a reflection corresponding to a $(u,v)$-separating fragment for some pair $u,v \in V$, and in this case \cref{lemma:mainsimple} implies that $F$ (and thus $F_1 \circ F_2$) is not equivalent to the trivial sequence.
\end{proof}

\subsection{Partial reflections and gluing}\label{subsection:gluing}

Our goal in this section is to understand how the sequences of partial $d$-reflections of two graphs with nonempty intersection ``fit together''. That is, given a pair of $d$-connected graphs $G_1$ and $G_2$ and (reduced) sequences of partial $d$-reflections $F_1$ and $F_2$, we would like to decide whether $F_1$ and $F_2$ are ``compatible'', in the sense that both ``come from'' a sequence $F$ of partial $d$-reflections of $G_1 \cup G_2$. It turns out that finding the right definitions for the phrases in quotation marks is nontrivial. We note that we will refer to the definitions and results of this (rather technical) subsection only once in the rest of the paper, namely in the proof of \cref{theorem:gluingkappa} (which is one of our main results). The proofs of the lemmas given in this section can be found in \cref{appendix:gluing}.   

We start by highlighting a notational issue. When discussing fragments, the complement of the fragment, as well as whether the fragment separates a pair of vertices, depend not only on the fragment itself, but on the underlying graph as well. This can cause ambiguities when we are considering a fragment $X$ in a subgraph $G_1$ of a graph $G$. In this case we shall sometimes explicitly write that we are taking the complement of $X$ in $G_1$, or that $X$ is $(u,v)$-separating in $G_1$. If not stated explicitly, then we shall always mean these statements with respect to the graph that was referenced at the definition of the fragment; e.g., if $X$ was introduced as a fragment of $G_1$, we shall implicitly take complements within $G_1$.

Let $G$ be the union of the graphs $G_1 = (V_1,E_1)$ and $G_2 = (V_2,E_2)$ and let $X$ be a $d$-fragment of $G_1$ and $Y$ a $d$-fragment of $G_2$. We say that $X$ and $Y$ are \emph{compatible} if there exists a $d$-fragment $Z$ of $G$ such that $X = Z \cap V_1$ or $\overline{X} = Z \cap V_1$ and $Y = Z \cap V_2$ or $\overline{Y} = Z \cap V_2$, where complements are taken in $G_1$ and $G_2$, respectively.
The following lemma gives an equivalent description of compatible pairs of $d$-fragments.

\begin{lemma}\label{lemma:compatiblefragmentchar}
    Let $G = (V,E)$ be the union of the graphs $G_1 = (V_1,E_1), G_2 = (V_2,E_2)$. Let $X$ be a $d$-fragment of $G_1$ and $Y$ a $d$-fragment of $G_2$. The following are equivalent. 
    \begin{enumerate}
        \item $X$ and $Y$ are compatible.
        \item We have $N_{G_1}(X) = N_{G_2}(Y)$, and for every pair of vertices $u,v \in V_1 \cap V_2$, $X$ is $(u,v)$-separating (in $G_1$) if and only if $Y$ is $(u,v)$-separating (in $G_2$).
    \end{enumerate}
    Moreover, if $X$ and $Y$ are compatible and $Z$ is as in the definition of compatibility, then $N_G(Z) = N_{G_1}(X)$.
\end{lemma}

Now let $G$ be the union of the $d$-connected graphs $G_1$ and $G_2$ and let $F_1 = (\sep{X_1},\ldots,\sep{X_k})$ and $F_2 = (\sep{Y_1},\ldots,\sep{Y_l})$ be reduced sequences of partial $d$-reflections of $G_1$ and $G_2$, respectively. We say that $F_1$ and $F_2$ are \emph{compatible} if there is an index $0 \leq r \leq \min(k,l)$ for which the following holds, after possibly reordering $X_1,\ldots,X_k$ and $Y_1,\ldots,Y_l$ and taking complements in $G_1$ and $G_2$, respectively.
\begin{itemize}
    \item $X_i$ and $Y_i$ are compatible for every $i \in \{1,\ldots,r\}$.
    \item $X_i$ is disjoint from $V_2$ for every $i \in \{r+1,\ldots,k\}$.
    \item $Y_j$ is disjoint from $V_1$ for every $j \in \{r+1,\ldots,l\}$.
\end{itemize}
Note that if either of $F_1$ and $F_2$ is trivial, then necessarily $r = 0$ and the first condition is vacuous.

Let $u,v \in V_1 \cap V_2$. We say that the pair $\{u,v\}$ \emph{strongly separates} $F_1$ and $F_2$ if it satisfies the following.
\begin{itemize}
    \item There is at least one fragment among $X_1,\ldots,X_k$ that is $(u,v)$-separating in $G_1$, or at least one fragment among $Y_1,\ldots,Y_l$ that is $(u,v)$-separating in $G_2$. 
    \item If $X_i$ and $Y_j$ are both $(u,v)$-separating, then $N_{G_1}(X_i) \neq N_{G_2}(Y_j)$, for every $i \in \{1,\ldots,k\}$ and $j \in \{1,\ldots,l\}$.
\end{itemize}
See \cref{figure:stronglyseparating} for an example demonstrating this notion.

\begin{figure}[t]
    \centering
        \includegraphics[]{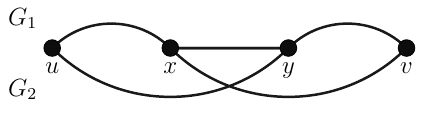}
        \caption{An example demonstrating the notion of strongly separating vertex pairs. Here $d = 1,$ $V(G_1) = V(G_2) = \{u,x,y,v\}$, $E(G_1) = \{ux,xy,yv\} $ and $E(G_2) = \{uy,xy,xv\}$. Let $X_1 = \{y,v\}, Y_1 = \{x,v\}$ and $X_2 = Y_2 = \{v\},$ and let $F_1 = (\sep{X_1}, \sep{X_2})$ and $F_2 = (\sep{Y_1},\sep{Y_2})$. Then the vertex pair $\{u,y\}$ strongly separates $F_1$ and $F_2$, since $X_1$ is $(u,y)$-separating (in $G_1$), while the fragments $X_2,Y_1,Y_2$ are not $(u,y)$-separating. In contrast, $\{u,v\}$ does not strongly separate $F_1$ and $F_2$, since $X_1$ and $Y_2$ are both $(u,v)$-separating with $N_{G_1}(X_1) = N_{G_2}(Y_2) = \{x\}$.}
        \label{figure:stronglyseparating}
\end{figure}

The following three lemmas clarify the reason for defining compatible pairs of sequences and strongly separating pairs of vertices. \cref{lemma:compatiblesequencechar} states that two sequences are either compatible, or there is a pair of vertices that strongly separates them. \cref{lemma:compatiblesequences} says, roughly, that if $F_1$ and $F_2$ are compatible, then they ``come from'' some sequence of partial $d$-reflections $F$ of $G_1 \cup G_2$. Finally, \cref{lemma:stronglyseparating} states that if the pair $\{u,v\}$ strongly separates $F_1$ and $F_2$, then $F_1$ and $F_2$ ``act differently'' on $\{u,v\}$, and thus they cannot come from some sequence of partial $d$-reflections of $G_1 \cup G_2$.

\begin{lemma}\label{lemma:compatiblesequencechar}
    Let $G = (V,E)$ be the union of the $d$-connected graphs $G_1 = (V_1,E_1), G_2 = (V_2,E_2)$ and let $F_1 = (\sep{X_1},\ldots,\sep{X_k})$ and $F_2 = (\sep{Y_1},\ldots,\sep{Y_l})$ be (possibly trivial) reduced sequences of partial $d$-reflections of $G_1$ and $G_2$, respectively. Then
    $F_1$ and $F_2$ are compatible if and only if there is no pair of vertices in $V_1 \cap V_2$ that strongly separates them.
\end{lemma}

\begin{lemma}\label{lemma:compatiblesequences}
    Let $G = (V,E)$ be the union of the $d$-connected graphs $G_1 = (V_1,E_1), G_2 = (V_2,E_2)$ and let $F_1 = (\sep{X_1},\ldots,\sep{X_k})$ and $F_2 = (\sep{Y_1},\ldots,\sep{Y_l})$ be (possibly trivial) reduced sequences of partial $d$-reflections of $G_1$ and $G_2$, respectively. Let $(G,p)$ be a generic realization in $\RR^d$ and let $(G_1,p_1)$ and $(G_2,p_2)$ denote the respective subframeworks. If $F_1$ and $F_2$ are compatible, then there exists a sequence of partial $d$-reflections $F = (\sep{Z_1},\ldots,\sep{Z_t})$ of $G$ such that the subframework of $(G,F(p))$ corresponding to $G_i$ is congruent to $(G_i,F_i(p_i))$, for each $i \in \{1,2\}$. 
\end{lemma}

The following is our main technical lemma.

\begin{lemma}\label{lemma:stronglyseparating}
Let $G = (V,E)$ be the union of the $d$-connected graphs $G_1 = (V_1,E_1), G_2 = (V_2,E_2)$ and let $F_1 = (\sep{X_1},\ldots,\sep{X_k})$ and $F_2 = (\sep{Y_1},\ldots,\sep{Y_l})$ be (possibly trivial) reduced sequences of partial $d$-reflections of $G_1$ and $G_2$, respectively. Let $u,v \in V_1 \cap V_2$ be a pair of vertices such that $\{u,v\}$ strongly separates $F_1$ and $F_2$.
Let $(G,p)$ be a generic realization in $\RR^d$ and let $(G_1,p_1)$ and $(G_2,p_2)$ denote the respective subframeworks. Then $\norm{F_1(p_1)(u) - F_1(p_1)(v)} \neq \norm{F_2(p_2)(u) - F_2(p_2)(v)}$.
\end{lemma}
In the case when $G = G_1 = G_2$ and $F_2$ is trivial, we recover \cref{lemma:mainsimple}. In fact, in this case our proof (given in \cref{appendix:gluing}) reduces exactly to the proof we gave for \cref{lemma:mainsimple}.

\section{Local connectivity and globally linked pairs}
\label{section:3char}

Our main goal in this paper is the investigation of the following graph class. We say that a graph $G = (V,E)$ is \emph{\mainclass[d]} if $|V| \geq d+1$, $G$ is rigid in $\RR^d$ and $G$ satisfies the following property:
\begin{equation}\label{eq:kappa}\tag{$*$}
    \{u,v\} \text{ is globally linked in } G \text{ in } \RR^d \Longleftrightarrow uv \in E \text{ or } \kappa(G,u,v) \geq d+1, \hspace{1em} \forall u,v \in V.
\end{equation}
Note that the ``$\Rightarrow$'' direction in \cref{eq:kappa} always holds by \cref{lemma:globallylinkedclosure}. Also note that \cref{eq:kappa} is equivalent to the condition that the globally rigid clusters of $G$ in $\RR^d$ are precisely the $(d+1)$-blocks of $G$. Since $G$ and  $\glclosure(G)$ have the same $(d+1)$-blocks (by \cref{corollary:globallylinkedclosure}) and the same globally rigid clusters in $\RR^d$, we obtain the following result which will be useful throughout this section.
\begin{lemma}\label{lemma:mainclassglobalclosure}
    A graph $G$ is \mainclass[d] if and only if $\glclosure(G)$ is \mainclass[d].
\end{lemma}

If a graph is globally rigid in $\RR^d$ on at least $d+1$ vertices, then it is \mainclass[d]; in fact, a \mainclass[d] graph on at least $d+2$ vertices is globally rigid if and only if it is $(d+1)$-connected. But not all \mainclass[d] graphs are globally rigid in $\RR^d$, as shown by \cref{figure:mainclassexample}. 
It is not difficult to construct graphs that satisfy \cref{eq:kappa} but are nonrigid, even if we further assume $d$-connectivity; for example, cycles of length at least four have this property in the $d = 2$ case. 

\begin{figure}[h]
    \centering
    \begin{subfigure}[b]{0.49\linewidth}
    \centering
        \includegraphics[]{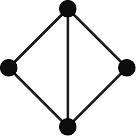}
        \caption{}
    \end{subfigure}
    \begin{subfigure}[b]{0.49\linewidth}
        \centering
        \includegraphics[]{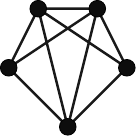}
        \caption{}
    \end{subfigure}
    \caption{The graph obtained from $K_{d+2}$ by deleting an edge is \mainclass[d], but not globally rigid in $\RR^d$. Here \textit{(a)} $d = 2$, \textit{(b)} $d = 3$.}
    \label{figure:mainclassexample}
\end{figure}

An immediate consequence of the definition of being \mainclass[d]\ is that the property of ``being globally linked'' is generic in \mainclass[d]\ graphs, in the sense that if there is some generic $d$-dimensional realization of $G$ in which the pair $\{u,v\}$ is globally linked, then $\{u,v\}$ is globally linked in every generic $d$-dimensional realization.

We close this subsection by observing that a graph on at least two vertices is \mainclass[1] if and only if it is connected. Indeed, it is well-known that in dimension one every graph satisfies \cref{eq:kappa}, and a graph is rigid in $\RR^1$ if and only if it is connected.

\subsection{Characterizations}

We start by giving three characterizations of \mainclass[d] graphs: one involving the globally linked closure of the graph, another in terms of the realizations equivalent to a generic realization in $\RR^d$, and a third in terms of the number of such realizations. We shall need the following structural result about rigid graphs in $\RR^d$, which we believe is interesting on its own right. The special case when $d=2$ was verified in \cite[Lemma 3.6]{jackson.jordan_2005}.

\begin{theorem}\label{theorem:rigidnoncrossing}
Let $G=(V,E)$ be a graph on at least $d+1$ vertices. If $G$ is rigid in $\RR^d$, then the $d$-separators of $G$ are pairwise noncrossing.
\end{theorem}

\begin{proof}
Recall that rigid graphs in $\RR^d$ on at least $d+1$ vertices are $d$-connected. Suppose, for a contradiction, that $G$ contains a pair $S_1,S_2$ of crossing $d$-separators. We show that this implies $r_d(G)< d|V|-\binom{d+1}{2}$, so $G$ is not rigid in $\RR^d$ by \cref{theorem:gluck}.
Since $S_i$ is a $d$-separator, there is a partition of $V$ into three nonempty sets $A_i,S_i,B_i$ such that there are no edges from $A_i$ to $B_i$, for $i \in \{1,2\}$. Moreover, since $S_1$ and $S_2$ are crossing, we may assume that the sets $X = S_1 \cap A_2$, $Y=S_1\cap B_2$, $P=S_2\cap A_1$, $Q=S_2\cap B_1$ are all nonempty. Let $C=S_1\cap S_2$ and let us denote the cardinality of these sets by $x=|X|, y=|Y|, p=|P|, q=|Q|$, and $c=|C|$. See \cref{figure:rigidnoncrossing}(a).

\begin{figure}[t]
    \centering
    \captionsetup[subfigure]{oneside,margin={0.6cm,0cm}}
    \begin{subfigure}[b]{0.49\linewidth}
    \centering
         \includegraphics[]{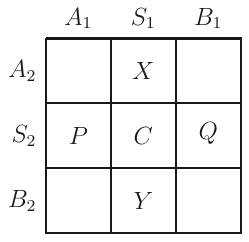}
        \caption{}  
    \end{subfigure}
    \captionsetup[subfigure]{oneside,margin={0.0cm,0cm}}
    \begin{subfigure}[b]{0.49\linewidth}
        \centering
        \includegraphics[]{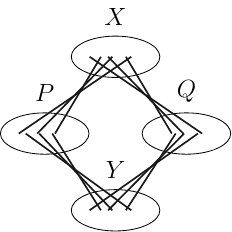}
        \caption{}
    \end{subfigure}
    \caption{An illustration of the different vertex sets appearing in the proof of \cref{theorem:rigidnoncrossing}.}
    \label{figure:rigidnoncrossing}
\end{figure}

We start by simplifying the problem in three steps. First, we may assume that the subgraphs of $G$ induced by $X\cup C\cup P$, $P\cup C\cup Y$, $Y\cup C\cup Q$, and $Q\cup C\cup X$, respectively, are complete. Indeed, adding new edges to these subgraphs preserves $d$-connectivity and the property that $S_1,S_2$ are crossing $d$-separators. Moreover, it cannot decrease the rank of the graph.
Then we observe that it suffices to show that $G[S_1\cup S_2]$ is not rigid in $\RR^d$, since each component of $G-(S_1\cup S_2)$ is attached to the rest of the graph along a complete subgraph, so \cref{lemma:rigidgluing} implies that if $G$ was rigid in $\RR^d$, then so would be $G[S_1 \cup S_2]$.
Thus, we may assume that $V = S_1 \cup S_2$. Finally, let $d'=d-c$. We can observe that it suffices to show that $G-C$ is nonrigid in $\RR^{d'}$. Indeed, $G-C$ is $d'$-connected, and
$S_1-C$ and $S_2-C$ are crossing $d'$-separators in $G-C$. Furthermore, since the
vertices of $C$ are connected to every other vertex in $G$, \cref{theorem:coning} implies that $G-C$ is nonrigid in $\RR^{d'}$ if and only if $G$ is nonrigid in $\RR^d$. Hence we may assume that $C=\varnothing$ and $d = d'$. See \cref{figure:rigidnoncrossing}(b).

By symmetry, we may assume that $x\geq y,$ $q\geq p$ and $p\geq y$.
This gives $x+q\geq x + y = d$, and thus \cref{theorem:gluck} implies that $r_d(G[X\cup Q]) = d(x+q)-\binom{d+1}{2}$.

We consider two cases in the rest of the proof. First suppose that $p=y$. Then we have $q = x = d-y$, and thus
\begin{align*}
|E| - |E(G[X\cup Q])| =  px + py + yq + \binom{p}{2} + \binom{y}{2}
&= y(2d-y)+y(y-1)=y(2d-1),
\end{align*}
which implies
\begin{align*}
r_d(G) &\leq r_d(G[X\cup Q]) + |E| - |E(G[X\cup Q])| = r_d(G[X\cup Q])+y(2d-1) \\ &= d(x+q) - \binom{d+1}{2}+ d(p+y) -y 
= d|V|-\binom{d+1}{2}-y,
\end{align*}
where the first inequality follows from the submodularity of $r_d$.
Since $y\geq 1$, this implies that $G$ is not rigid in $\RR^d$.

Next, suppose that $p \geq y+1$. Observe that $G[X\cup P]$ is a complete graph on at least $d+1$ vertices. Let $K$ be a complete subgraph of $G[X \cup P]$ on $d+1$ vertices that contains $X$. Now $G[X\cup P]$ has a rigid spanning subgraph that can be obtained from $K$ by adding  $x+p-d-1=p-y-1$ vertices of degree $d$.\footnote{It is well-known that adding a new vertex of degree $d$ preserves rigidity in $\RR^d$.} Therefore the rest of the edges in $G[X\cup P]$ can be deleted without decreasing the rank of $G$. After deleting these edges we have \[|E(G[X \cup P])| = \binom{x}{2} + x(d + 1 - x) + \binom{d+1-x}{2} + d(p-y-1),\]
where the first three terms count the number of edges in $K$. Hence we have
\begin{align*}
r_d(G) - r_d(G[X\cup Q]) &\leq  |E| - |E(G[X \cup Q])| \\ &= x(d + 1 - x) + \binom{d+1-x}{2} + d(p-y-1) + (p+q)y + \binom{y}{2} \\ &=  (d-y)(y+1) + \binom{y+1}{2} + d(p-y-1) + dy + \binom{y}{2} \\ &= d(p+y) + \binom{y}{2} + \binom{y+1}{2} - y(y+1) = d(p+y) - y.
\end{align*}
Since $y \geq 1$, it follows that
\[r_d(G) < r_d(G[X\cup Q])+d(p+y) = d(x+q) - \binom{d+1}{2}+ d(p+y) = d|V|-\binom{d+1}{2},\]so $G$ is not rigid in $\RR^d$.
\end{proof}

\begin{theorem}\label{theorem:closurechar}
A graph $G = (V,E)$ is \mainclass[d] if and only if $\glclosure(G)$ is $d$-connected, every $d$-separator of $\glclosure(G)$ is a clique, and every $(d+1)$-block of $\glclosure(G)$ is a clique. 
\end{theorem}
\begin{proof}
Let us first assume that $G$ is \mainclass[d]. Since $G$ is rigid, it is $d$-connected, and thus so is $\glclosure(G)$ by \cref{corollary:globallylinkedclosure}. Also, if $S$ is a $d$-separator of $G$ and $u,v \in S$ is a nonadjacent pair of vertices, then $\kappa(G,u,v) \geq d+1$, for otherwise there would be a $d$-separator in $G$ that crosses $S$, contradicting \cref{theorem:rigidnoncrossing}. 
Now \cref{eq:kappa} implies
that every $d$-separator of $\glclosure(G)$ is a clique. It is also immediate from \cref{eq:kappa} that every $(d+1)$-block of $\glclosure(G)$ is a clique.

Now we prove the converse. By \cref{lemma:mainclassglobalclosure} it is enough to prove that $\glclosure(G)$ is \mainclass[d], so we may assume that $G = \glclosure(G)$. Since $G$ is $d$-connected, it has at least $d+1$ vertices. The $(d+1)$-blocks of $G$ are all complete, so $G$ trivially satisfies \cref{eq:kappa}. All that is left to show is that $G$ is rigid in $\RR^d$. This is clear if $G$ is complete, while otherwise we can proceed by induction on the number of vertices by cutting along a $d$-separator and using \cref{lemma:rigidgluing}; note that if $X$ is a $d$-fragment of $G$, then by \cref{lemma:dblock} the subgraphs $G[X \cup N_G(X)]$ and $G[V-X]$ also satisfy the assumptions made on $G$, so we can indeed use induction. We omit the details.
\end{proof}

The proof of \cref{theorem:closurechar} also shows that in the definition of \mainclass[d] graphs, the condition that the graph is rigid in $\RR^d$ may be weakened.
\begin{corollary}\label{corollary:weakerdef}
A graph $G$ is \mainclass[d] if and only if it is $d$-connected, its $d$-separators are pairwise noncrossing, and it satisfies \cref{eq:kappa}.    
\end{corollary}
\begin{proof}
    Necessity follows from \cref{theorem:rigidnoncrossing}. To prove sufficiency we can repeat the first paragraph of the preceding proof to obtain that $\glclosure(G)$ is $d$-connected, every $d$-separator of $\glclosure(G)$ is a clique, and every $(d+1)$-block of $\glclosure(G)$ is a clique. By \cref{theorem:closurechar} this means that $G$ is \mainclass[d].
\end{proof}

\begin{theorem}\label{theorem:partialreflectionchar}
A graph $G = (V,E)$ on at least $d+1$ vertices is \mainclass[d] if and only if for every generic realization $(G,p)$ in $\RR^d$, every equivalent realization $(G,q)$ is congruent to $(G,F(p))$ for some reduced sequence of partial $d$-reflections $F$ of $G$.
\end{theorem}
\begin{proof}
We first show that if $G$ is not \mainclass[d], then there is a generic framework $(G,p)$ and an equivalent framework $(G,q)$ that is not congruent to any framework of the form $(G,F(p))$. If $G$ is not rigid in $\RR^d$, then by a result of Asimow and Roth \cite[Proposition 1]{asimow.roth_1978} any generic $(G,p)$ can be continuously flexed, so there are continuum-many congruence classes of frameworks equivalent to $(G,p)$. On the other hand, there is only a countable number of sequences of partial $d$-reflections. It follows that there is an equivalent framework $(G,q)$ that is not congruent to $(G,F(p))$ for any sequence of partial $d$-reflections $F$. If $G$ is rigid but not \mainclass[d], then there is a pair of vertices $u,v \in V$ such that $\kappa(G,u,v) \geq d+1$ but $\{u,v\}$ is not globally linked in $G$ in $\RR^d$. This means that there is a generic framework $(G,p)$ and an equivalent framework $(G,q)$ with $\norm{p(u) - p(v)} \neq \norm{q(u) - q(v)}$. Since there are no $(u,v)$-separating $d$-fragments in $G$, $\norm{p(u) - p(v)} = \norm{F(p)(u) - F(p)(v)}$ for any sequence of partial $d$-reflections $F$, so $(G,q)$ cannot be congruent to $(G,F(p))$.

Now let $G$ be \mainclass[d]. We may replace $G$ by $\glclosure(G)$ since the two graphs have the same equivalence classes of realizations by \cref{lemma:equivclasses}, as well as the same $d$-separators (and thus sequences of partial $d$-reflections) by \cref{corollary:globallylinkedclosure}.
So let us assume $G = \glclosure(G)$. Note that by \cref{theorem:closurechar} every $d$-separator and every $(d+1)$-block of $G$ is a clique. We proceed by induction on the number of vertices. If $G$ is globally rigid in $\RR^d,$ then we are done. In particular, this holds if $|V| = d+1$. Now let us suppose that $|V| > d+1$ and $G$ is not globally rigid in $\RR^d$. Since $G$ is \mainclass[d], this means that it is not $(d+1)$-connected. Let $X$ be an inclusion-wise minimal $d$-fragment of $G$, let $S = N_G(X)$ be the corresponding $d$-separator, and let $G_1$ and $G_2$ denote the subgraphs induced by $X \cup S$ and $G-X$, respectively. On the one hand, by \cref{corollary:weakerdef} there are no crossing $d$-separators in $G$, 
so by \cref{lemma:inclusionwiseminimalfragment} $X \cup S$ is a $(d+1)$-block of $G$. Hence $X \cup S$ is a clique in $G$, and in particular $G_1$ is globally rigid in $\RR^d$. On the other hand, since $S$ is a clique in $G$, we can apply \cref{lemma:dblock} to deduce that $G_2$ is $d$-connected, and every $d$-separator and every $(d+1)$-block of $G_2$ is a clique in $G_2$. Thus by \cref{theorem:closurechar} we have that $G_2$ is \mainclass[d].

Consider a generic realization $(G,p)$ in $\RR^d$ and an equivalent realization $(G,q)$, and let $(G_i,p_i), (G_i,q_i)$ denote the respective subframeworks for $i \in \{1,2\}$. Since $G_1$ is globally rigid in $\RR^d$, $(G_1,q_1)$ is congruent to $(G_1,p_1)$. By applying a congruence sending $q_1$ to $p_1$ to all of $(G,q)$ we may suppose that $q_1=p_1$. 
From the induction hypothesis it follows that $(G_2,q_2)$ is congruent to a framework $(G_2,F'(p_2))$, where $F' = (\sep{X'_1},\ldots,\sep{X'_k})$ is a (possibly trivial) sequence of partial $d$-reflections of $G_2$. Since $S$ is a clique in $G_2$, $S - N_{G_2}(X_i')$ is contained in either $X_i'$ or $\overline{X'_i}$, for each $i \in \{1,\ldots,k\}$. By \cref{lemma:equivcomplement}, we may freely replace some of the fragments $X'_i$ by their complement $\overline{X'_i}$ in $G_2$; in this way, we obtain a framework that is congruent to $(G_2,F'(p_2))$. It follows that we may assume $S \subseteq X'_i \cup N_{G_2}(X'_i)$ for $i \in \{1,\ldots,k\}$, and thus $F'(p_2)(s) = p_2(s) = q_2(s)$ for every $s \in S$. 

Now the congruence $\alpha : \RR^d \rightarrow \RR^d$ between $(G_2,F'(p_2))$ and $(G_2,q_2)$ fixes the points $p(s), s \in S$. Since $(G,p)$ is generic, this is a set of $d$ affinely independent points, and thus by \cref{lemma:congruence}, $\alpha$ is either the identity map or the orthogonal reflection $\reflection{p(S)}$.
For each $i \in \{1,\ldots,k\}$ let $X_i$ be the fragment of $G$ defined by
\begin{equation*}
X_i = \begin{cases}
X'_i & N_{G_2}(X'_i) \neq S,\\
X'_i \cup X & N_{G_2}(X'_i) = S,
\end{cases}
\end{equation*}
and consider the sequence of partial $d$-reflections $F = (\sep{X_1},\ldots,\sep{X_k})$ of $G$. Then either $q = F(p)$ (if $\alpha$ is the identity), or $q = (\sep{X} \circ F)(p)$ (if $\alpha = \reflection{p(S)}$). In both cases $(G,q)$ is congruent (in fact, equal, under the assumptions made during the proof) to a framework that can be obtained from $(G,p)$ by a sequence of partial $d$-reflections of $G$.
\end{proof}

For a set of vertices $S \subseteq V$, let $b_G(S)$ denote the number of components of $G - S$ and let \[c_d(G) = \sum_{S \subseteq V, |S| = d}(b_G(S) - 1).\] For a generic $d$-dimensional realization $(G,p)$, let $c(G,p)$ denote the maximum number of pairwise noncongruent realizations that are equivalent to $(G,p)$. It is well-known that this number is finite if and only if $G$ is rigid in $\RR^d$.

Our final characterization of \mainclass[d] graphs is in terms of the value of $c(G,p)$ for every generic $d$-dimensional realization $(G,p)$. This generalizes \cite[Theorem 8.2]{jackson.etal_2006} and \cite[Theorem 15(ii)]{conferencepaper}, the analogous results in the case of $\mathcal{R}_2$-connected graphs and braced maximal outerplanar graphs, respectively. It is known that the graphs belonging to either of these families are \mainclass[2]; we shall return to these examples in \cref{subsection:examples}.

\begin{theorem}\label{theorem:numberofrealizations}
Let $G = (V,E)$ be a graph on at least $d+1$ vertices. For every $d$-dimensional generic realization of $G$ we have $c(G,p) \geq 2^{c_d(G)}$. Furthermore, $G$ is \mainclass[d] if and only if $c(G,p) = 2^{c_d(G)}$ holds for every generic $d$-dimensional realization $(G,p)$. 
\end{theorem}
\begin{proof}
If $G$ is not rigid in $\RR^d$, then $c(G,p)$ is infinite for every generic realization of $G$. Thus, let us assume that $G$ is rigid in $\RR^d$. In this case $G$ has no crossing $d$-separators by \cref{theorem:rigidnoncrossing}. Let us choose an arbitrary ordering $S_1,\ldots,S_k$ of the $d$-separators of $G$ and let us fix vertices $v_1,\ldots,v_k$ with $v_i \notin S_i$ but otherwise arbitrarily. Note that $2^{b_G(S_i)}$ is the number of $d$-fragments $X_i$ of $G$ with $N_G(X_i) = S_i$, plus $2$ (corresponding to the ``trivial $d$-fragments'' $X_i = \varnothing$ and $X_i = V - S_i$).  It follows that $2^{c_d(G)}$ counts the sequences $(X_1,\ldots,X_k)$ of sets such that either $X_i = \varnothing$ or $X_i$ is a $d$-fragment with $v_i \in X_i$ and $N_G(X_i) = S_i$. Every such sequence $\mathcal{S}$ corresponds uniquely to a reduced sequence of partial $d$-reflections $F_{\mathcal{S}}$ by deleting the occurrences of $\varnothing$ from the sequence, and every reduced sequence of partial $d$-reflections is equivalent to a sequence of this form by reordering and possibly taking the complements of some of the fragments included (here we use \cref{lemma:equivcomplement,lemma:equivnoncrossing}). Furthermore,  \cref{corollary:noncongruentreflections} implies that for different sequences $\mathcal{S}_1,\mathcal{S}_2$, we have $F_{\mathcal{S}_1} \not \partialeq F_{\mathcal{S}_2}$. Combined with \cref{lemma:equivgeneric}, we obtain that $c(G,p) \geq 2^{c_d(G)}$ holds for any generic realization $(G,p)$ of $G$, while \cref{theorem:partialreflectionchar} implies that equality holds if and only if $G$ is \mainclass[d]. 
\end{proof}

\subsection{Constructions} 

In this subsection we verify that two operations preserve the property of being \mainclass[d]: the first is edge addition, and the second is gluing along at least $d$ vertices.

\begin{theorem}\label{theorem:monotonicity}
Let $G = (V,E)$ be a graph and let $u$ and $v$ be a pair of nonadjacent vertices in $G$. If $G$ is \mainclass[d], then so is $G+uv$.
\end{theorem}
\begin{proof}
Since rigidity is preserved by edge addition, we only need to show that for any pair $\{x,y\}$ of nonadjacent vertices in $G+uv$, if $\kappa(G+uv,x,y) \geq d+1$, then $\{x,y\}$ is globally linked in $G+uv$ in $\RR^d$. If $\kappa(G,x,y) \geq d+1$ holds, then $\{x,y\}$ is globally linked in $G$ in $\RR^d$, and thus it is globally linked in $G+uv$ in $\RR^d$ as well. Thus, let us assume that $\kappa(G,x,y) \leq d $. This implies that every $(x,y)$-separating $d$-fragment in $G$ is also $(u,v)$-separating. Let $(G+uv,p)$ be a generic realization in $\RR^d$ and consider an equivalent framework $(G+uv,q)$. Now $(G,p)$ and $(G,q)$ are also equivalent, so by \cref{theorem:partialreflectionchar} there is a reduced sequence of partial $d$-reflections $F$ such that $(G,q)$ is congruent to $(G,F(p))$. For a contradiction, suppose that $F$ contains at least one reflection onto an $(x,y)$-separating fragment. This fragment is also $(u,v)$-separating, so by \cref{lemma:mainsimple} we have that \[\norm{p(u)-p(v)} \neq \norm{F(p)(u) - F(p)(v)} = \norm{q(u) - q(v)},\]contradicting the fact that $(G+uv,p)$ and $(G+uv,q)$ are equivalent. This shows that $F$ cannot contain reflections onto $(x,y)$-separating fragments, so we have $\norm{p(x) - p(y)} = \norm{q(x) - q(y)}$, as desired.
\end{proof}

\begin{theorem}\label{theorem:gluingkappa}
    Let $G = (V,E)$ be the union of the graphs $G_1 = (V_1,E_1)$ and $G_2 = (V_2,E_2)$. If $G_1$ and $G_2$ are \mainclass[d], then $G$ satisfies \cref{eq:kappa}. 
\end{theorem}
\begin{proof}
    Let us fix a pair of nonadjacent vertices $u,v \in V$. As we noted before, one direction of the equivalence in \cref{eq:kappa} is always satisfied. Thus we only need to show that if $\kappa(G,u,v) \geq d+1$, then $\{u,v\}$ is globally linked in $G$ in $\RR^d$.
    Let $(G,p)$ be a generic framework in $\RR^d$, let $(G,q)$ be an equivalent framework, and for each $i \in \{1,2\}$ let $(G_i,p_i)$ and $(G_i,q_i)$ denote the subframeworks of $(G,p)$ and $(G,q)$, respectively, corresponding to $G_i$. Since $G_i$ is \mainclass[d], \cref{theorem:partialreflectionchar} implies that there is a reduced sequence of partial $d$-reflections $F_i$ of $G_i$ such that $(G_i,q_i)$ and $(G_i,F_i(p_i))$ are congruent, for $i \in \{1,2\}$. 

    We claim that $F_1$ and $F_2$ are compatible. For if they are not, then \cref{lemma:compatiblesequencechar} implies that there is a pair of vertices $u,v \in V_1 \cap V_2$ that strongly separates $F_1$ and $F_2$, and hence by \cref{lemma:stronglyseparating} we have $\norm{F_1(p_1)(u) - F_1(p_1)(v)} \neq \norm{F_2(p_2)(u) - F_2(p_2)(v)}$. But this contradicts
    \begin{align*}
    \norm{F_1(p_1)(u) - F_1(p_1)(v)} &= \norm{q_1(u) - q_1(v)} \\ &= \norm{q(u) - q(v)} \\ &= \norm{q_2(u) - q_2(v)} = \norm{F_2(p_2)(u) - F_2(p_2)(v)}.
    \end{align*}
    It follows that $F_1$ and $F_2$ are compatible.

    By \cref{lemma:compatiblesequences} there is a sequence $F$ of partial $d$-reflections of $G$ such that the subframework of $(G,F(p))$ corresponding to $G_i$ is congruent to $(G_i,F_i(p_i))$, and thus to $(G_i,q_i)$, for each $i \in \{1,2\}$. Note that since $\kappa(G,u,v) \geq d+1$, there are no $(u,v)$-separating fragments in $F$ and hence $\norm{p(u) - p(v)} = \norm{F(p)(u) - F(p)(v)}$. If $u,v \in V_i$ for some $i \in \{1,2\},$ then we further have 
    \[\norm{p(u) - p(v)} = \norm{F(p)(u) - F(p)(v)} = \norm{q_i(u) - q_i(v)} = \norm{q(u) - q(v)},\]
    which (by the arbitrary choice of $(G,p)$ and $(G,q)$) shows that $\{u,v\}$ is globally linked in $G$ in $\RR^d$.

    The only remaining case is when $u \in V_1 - V_2$ and $v \in V_2 - V_1$, or $u \in V_2 - V_1$ and $v \in V_1 - V_2$; by symmetry, we may assume that it is the former. In this case $\kappa(G,u,v) \geq d+1$ implies that $|V_1 \cap V_2| \geq d+1$. Let $\alpha_i, i \in \{1,2\}$ denote the congruence that maps the subframework of $(G,F(p))$ corresponding to $G_i$ to $(G_i,q_i)$. Note that $\alpha_1$ and $\alpha_2$ agree on the points $\{F(p)(x), x \in V_1 \cap V_2\}$. Since $(G,p)$ is generic, so is $(G,F(p))$ by \cref{lemma:equivgeneric}, and hence $\{F(p)(x), x \in V_1 \cap V_2\}$ contains $d+1$ affinely independent points. It follows from \cref{lemma:congruence} that $\alpha_1 = \alpha_2$, and thus $(G,F(p))$ is congruent to $(G,q)$. Hence we have \[\norm{p(u) - p(v)} = \norm{F(p)(u) - F(p)(v)} = \norm{q(u) - q(v)},\] as desired.
\end{proof}

\begin{corollary}\label{theorem:maingluing}
    Let $G$ be the union of the graphs $G_1$ and $G_2$. If $G_1$ and $G_2$ are \mainclass[d] and $|V_1 \cap V_2| \geq d$, then $G$ is \mainclass[d].
\end{corollary}
\begin{proof}
    By \cref{theorem:gluingkappa}, $G$ satisfies \cref{eq:kappa}. Since $G_1$ and $G_2$ are \mainclass[d], they are rigid, and thus so is $G$ by \cref{lemma:rigidgluing}. Finally, since $G_1$ has at least $d+1$ vertices, so does $G$.
\end{proof}

\subsection{Examples}
\label{subsection:examples}

As we noted before, globally rigid graphs in $\RR^d$ are (trivially) \mainclass[d]. Using this fact and \cref{theorem:maingluing} we can show that the members of the following, substantially larger family of graphs are \mainclass[d] as well. Let $G$ be a graph on at least $d+1$ vertices. A \emph{(globally rigid) gluing construction (in $\RR^d$)} of $G$ is a sequence of graphs $G_1,\ldots,G_k = G$ such that for every $i \in \{1,\ldots,k\}$, either $G_i$ is a globally rigid graph in $\RR^d$ on at least $d+1$ vertices, or there are indices $1 \leq j < l < i$ such that $G_i$ is the union of $G_j$ and $G_l$ with $V(G_j) \neq V(G_i) \neq V(G_l)$ and $|V(G_j) \cap V(G_l)| \geq d$. When the dimension $d$ is clear from the context, we shall sometimes simply write that $G_1,\ldots,G_k$ is a gluing construction of $G$.

Now we can quickly prove one of our main results.
\begin{theorem}\label{theorem:gluingconstruction}
    Let $G = (V,E)$ be a graph that has a globally rigid gluing construction in $\RR^d$. Then $G$ is \mainclass[d]. In particular, $G$ is globally rigid in $\RR^d$ if and only if either $G = K_{d+1}$, or $G$ is $(d+1)$-connected.
\end{theorem}
\begin{proof}
    We prove by induction on $|V|$. The only graph on $d+1$ vertices that has a globally rigid gluing construction is the complete graph, which is \mainclass[d]. Thus, we may assume that $|V| \geq d+2$. Let $G_1,\ldots,G_k = G$ be a gluing construction of $G$. Again, if $G$ is globally rigid in $\RR^d$, then it is \mainclass[d]. Otherwise there are indices $1 \leq j < l < k$ such that $G$ is the union of $G_j$ and $G_l$, with $|V(G_j) \cap V(G_l)| \geq d$ and both $|V(G_j)|$ and $|V(G_l)|$ are smaller than $|V|$. Now $G_j$ and $G_l$ each have a globally rigid gluing construction in $\RR^d$, so by the induction hypothesis they are both \mainclass[d]. It follows by \cref{theorem:maingluing} that $G$ is also \mainclass[d].
\end{proof}

We can use \cref{theorem:gluingconstruction} to prove that the family of graphs that have a globally rigid gluing construction in $\RR^d$ is closed under edge addition, as well as under separating along a $d$-separator.

\begin{lemma}\label{lemma:gluingmonotonicity} Suppose that the graph $G = (V,E)$ has a globally rigid gluing construction in $\RR^d$.
\begin{enumerate}
    \item If $X$ is a $d$-fragment of $G$, then the subgraph of $G$ induced by $X \cup N_G(X)$ also has a globally rigid gluing construction in $\RR^d$.
    \item If $\{u,v\}$ is a pair of nonadjacent vertices in $G$, then $G+uv$ also has a globally rigid gluing construction in $\RR^d$.
\end{enumerate}
\end{lemma}
\begin{proof}
    We prove the two claims together by induction on $|V|$.  If $|V| = d+1$, then $G$ is necessarily a complete graph and both statements are vacuously true. Thus, let us assume that $|V| \geq d+2$. We first prove part \textit{(a)}. Let us use the notation $S = N_G(X)$ and $G' = G[X \cup S]$. Let $G_1, \ldots, G_k = G$ be a gluing construction of $G$. Since $G$ is not $(d+1)$-connected, it is not globally rigid in $\RR^d$, so there are indices $j,l \in \{1,\ldots, k\}$ such that $G = G_j \cup G_l$, where $|V(G_j) \cap V(G_l)| \geq d$ and $|V(G_j)|,|V(G_l)| < |V|$. By possibly swapping $j$ and $l$, we may assume that $V(G_l) \cap X$ is nonempty. 
    
    Consider $G'_l = G_l[V(G_l) \cap (X \cup S)]$. If $V(G_l) \subseteq X \cup S$, then $G'_l = G_l$, and in particular $G'_l$ has a gluing construction. On the other hand, if $V(G_l) \cap \overline{X}$ is nonempty, then since $G_l$ is $d$-connected we must have $S \subseteq V(G_l)$, and thus $X \cap V(G_l)$ is a $d$-fragment of $G_l$ with $N_{G_l}(X \cap V(G_l)) = S$. By induction, $G'_l$ has a gluing construction in this case as well.

    Now let us consider $G_j$. If $V(G_j) \cap X$ is nonempty, then by the same reasoning as in the previous paragraph we obtain that $G'_j = G_j[V(G_j) \cap (X \cup S)]$ has a gluing construction. Since $G' = G'_j \cup G'_l$, $G'$ also has a rigid gluing construction in this case. Thus, the only case left to consider is when $V(G_j) \cap X$ is empty. Since $|V(G_j) \cap V(G_l)| \geq d$, either we have $V(G_j) \cap V(G_l) = S$, or $V(G_l) \cap \overline{X}$ is nonempty. In both cases $S \subseteq V(G_l)$ holds. It follows that $G'_l$ is a spanning subgraph of $G'$, i.e., $G'$ can be obtained from $G'_l$ using edge additions. By using part \textit{(b)} of the inductive hypothesis it follows that $G'$ has a gluing construction, as desired.

    Next, we prove part \textit{(b)}. Note that $G$ is $d$-connected. By \cref{theorem:gluingconstruction}, $G$ is \mainclass[d], and thus by \cref{theorem:monotonicity} so is $G+uv$. It follows that if $G+uv$ is $(d+1)$-connected, then it is globally rigid in $\RR^d$, so in particular it has a gluing construction. Let us thus suppose that $G+uv$ is not $(d+1)$-connected, and let $X$ be a $d$-fragment of $G+uv$. Now $X$ is also a $d$-fragment of $G$, so by part \textit{(a)}, both $G[X \cup N_G(X)]$ and $G[V-X]$ have a gluing construction. Let $G_1$ and $G_2$ be the graphs obtained from $G[X \cup N_G(X)]$ and $G[V-X]$, respectively, by adding the edge $uv$ if the corresponding graph spans both $u$ and $v$ (and doing nothing otherwise). By induction, both $G_1$ and $G_2$ have a gluing construction. Since $G+uv$ arises by gluing $G_1$ and $G_2$ along $N_G(X)$, $G+uv$ also has a gluing construction, as claimed. 
\end{proof}

As a subexample we consider chordal graphs. A graph is \emph{chordal} if it contains no induced cycle of length at least four. It is well-known that in a chordal graph, every minimum vertex cut (in fact, every minimal separator) is a clique. 

\begin{theorem}\label{theorem:chordalgendtree}
    Let $G = (V,E)$ be a graph. If $G$ is $d$-connected and chordal, then $G$ has a globally rigid gluing construction in $\RR^d$. Consequently, $G$ is \mainclass[d].
\end{theorem}
\begin{proof}
    We prove by induction on $|V|$. If $|V| = d+1$, then $G$ is a complete graph, so it has a (trivial) gluing construction. Thus, we may assume that $|V| \geq d+2$. Let $X$ be a fragment corresponding to a minimum vertex cut of $G$ and let us define $G_1 = G[X \cup N_G(X)]$ and $G_2 = G[V - X]$. Since $G$ is chordal and $N_G(X)$ is a minimum vertex cut, it is a clique in $G$, and this implies that $G_1$ and $G_2$ are also $d$-connected and chordal. By induction, both $G_1$ and $G_2$ have a gluing construction. Since $G$ arises by gluing $G_1$ and $G_2$ along $N_G(X)$, and the latter set has size at least $d$, $G$ also has a gluing construction. 
\end{proof}

We note that \cref{theorem:chordalgendtree} recovers the result of Alfakih \cite{alfakih_2012} stating that $(d+1)$-connected chordal graphs are globally rigid in $\RR^d$.
In \cite{conferencepaper} we showed, using our current terminology, that graphs containing a maximal outerplanar spanning subgraph are \mainclass[2]. Since maximal outerplanar graphs are $2$-connected chordal graphs, by combining \cref{theorem:chordalgendtree} and \cref{theorem:monotonicity} we recover (and substantially generalize) this result. We note that 
the families of maximal outerplanar graphs and $d$-connected chordal graphs are not closed under edge addition. 

In \cite{cruickshank.etal_2022}, the authors introduce graphs with the \emph{strong cleavage property} with respect to some fixed dimension $d \geq 3$. When $d \geq 4$, the definition is that the graph is $d$-connected, its $d$-separators are cliques and its $(d+1)$-blocks globally rigid in $\RR^d$.\footnote{The authors use a different definition for $(d+1)$-blocks, but for $d$-connected graphs in which the $d$-separators are cliques the two notions coincide.} The authors also show that for each $d$, this graph family is closed under edge addition, as well as under gluing along at least $d$ vertices, provided that among the identified vertices there is a set of size $d$ that induces a clique in both graphs. It is not difficult to see that for $d \geq 4$, graphs with the strong cleavage property have a globally rigid gluing construction in $\RR^d$, and thus are \mainclass[d]. Using this and \cref{theorem:maingluing} we can slightly strengthen their gluing result by showing that gluing two graphs with the strong cleavage property along at least $d+1$ vertices (or along $d$ vertices that induce a clique in both graphs) results in a graph with the strong cleavage property.

Our other main example of \mainclass[d] graphs, in the $d=2$ case, are $\mathcal{R}_2$-connected graphs. This is based on the following result.

\begin{theorem}\cite[Theorem 5.7]{jackson.etal_2006}\label{theorem:jacksonjordanszabadka}
    Let $G = (V,E)$ be an $\mathcal{R}_2$-connected graph and let $u,v \in V$ be a pair of vertices. The pair $\{u,v\}$ is globally linked in $G$ in $\RR^2$ if and only if $\kappa(G,u,v) \geq 3$.
\end{theorem}

It is well-known that $\mathcal{R}_2$-connected graphs are (redundantly) rigid in $\RR^2$ (see, e.g., \cite{jackson.jordan_2005}). Thus, \cref{theorem:jacksonjordanszabadka} implies that $\mathcal{R}_2$-connected graphs are \mainclass[2]. The analogous statement is not true in $d \geq 3$ dimensions: in this case, $\mathcal{R}_d$-connected graphs can be nonrigid in $\RR^d$, as demonstrated by the well-known ``double banana'' graph, see \cref{figure:doublebanana}(a). In the next section, we shall also give an example of an $\mathcal{R}_3$-connected graph that is rigid in $\RR^3$ but is not \mainclass[3].

\begin{figure}[h]
    \centering
    \begin{subfigure}[b]{0.48\linewidth}
    \centering
                \includegraphics[]{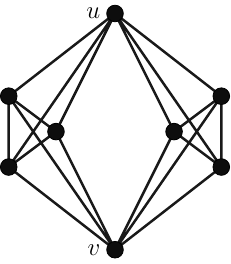}
    \caption{}
    \end{subfigure}
    \begin{subfigure}[b]{0.48\linewidth}
        \centering
                   \includegraphics[]{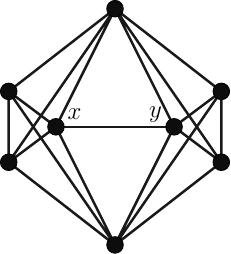}
        \caption{}
    \end{subfigure}
    \caption{\textit{(a)} The ``double banana'' graph $B$, and \textit{(b)} the graph $B' = B + xy.$}
    \label{figure:doublebanana}
\end{figure}

Finally, let us consider the double banana graph $B$ in some more detail. This graph is the union of two subgraphs $G_1 = (V_1,E_1),G_2 = (V_2,E_2)$, each isomorphic to $K_5$ minus an edge. Let $V_1 \cap V_2 = \{u,v\}$ denote the ``hinge vertices'' in $B$. Both $G_1$ and $G_2$ are chordal and $3$-connected, so they are \mainclass[3] by \cref{theorem:chordalgendtree}. Thus \cref{theorem:gluingkappa} applies, implying that $B$ satisfies \cref{eq:kappa}, and in particular that $\{u,v\}$ is globally linked in $B$ in $\RR^3$. However, $B$ is not \mainclass[3], since it is nonrigid in $\RR^3$.
Adding an edge $xy$ to $B$ other than $uv$ results in a rigid graph $B'$ (\cref{figure:doublebanana}(b)). It is not difficult to see that $\text{cl}^*_3(B') = B' + uv$, and the latter is a $3$-connected chordal graph. Thus \cref{theorem:chordalgendtree} implies that $B'$ is \mainclass[3].

\section{Concluding remarks}\label{section:concluding}

\subsection{Partial reflections and noncrossing separators}

In \cref{section:partialreflections} we introduced sequences of partial reflections and analyzed them in some detail. We showed that in $d$-connected graphs with no crossing $d$-separators, partial reflections have a simple structure: every sequence of partial $d$-reflections is equivalent to a reduced sequence of partial $d$-reflections (\cref{corollary:reducedsequence}) and the reduced sequences are pairwise nonequivalent, apart from trivialities (\cref{corollary:noncongruentreflections}).

If we allow crossing $d$-separators, the situation becomes more complicated. Although we did not pursue this direction, we believe that the following is true: if $X$ and $Y$ are crossing $d$-fragments in a $d$-connected graph, then the sequences of partial reflections $\mathcal{R}_X, \mathcal{R}_X \circ \mathcal{R}_Y, \mathcal{R}_X \circ \mathcal{R}_Y \circ \mathcal{R}_X, \ldots,$ are pairwise nonequivalent. Note that this would lead to an alternative, geometric proof of \cref{theorem:rigidnoncrossing}, since given any generic realization of a rigid graph, the maximum number of pairwise noncongruent, equivalent realizations is finite.

In fact, this is known in the $d=2$ case.
Partial $2$-reflections of the four-cycle $C_4$ have been studied (under a different name) in the context of \emph{Darboux' porism}. As a consequence, for each reduced sequence of partial $2$-reflections $F$ of $C_4$, a characterization is available for the space of realizations $(C_4,p)$ that are mapped to a congruent copy by $F$. See \cite{izmestiev_2020} and references therein.

\subsection{Globally linked pairs and \texorpdfstring{$\mathcal{R}_d$}{Rd}-connectivity}

Jackson, Jordán and Szabadka conjectured that \cref{theorem:jacksonjordanszabadka} gives a complete characterization of globally linked pairs in $\RR^2$ in the following sense.

\begin{conjecture}\cite[Conjecture 5.9]{jackson.etal_2006}\label{conjecture:jacksonjordanszabadka} The pair of vertices $\{u, v\}$ is globally linked in a graph $G = (V, E)$ in $\RR^2$ if and
only if either $uv \in E$ or there is an $\mathcal{R}_2$-component $H = (V',E')$ of $G$ with $u,v \in V'$ and $\kappa(H,u,v) \geq 3$.
\end{conjecture}

The analogues of \cref{theorem:jacksonjordanszabadka} and \cref{conjecture:jacksonjordanszabadka} in $d \geq 3$ dimensions are not true, as shown by \cref{figure:6ring}. However, it is still possible that the ``necessity'' part of their conjecture holds in these dimensions.
\begin{conjecture}\label{conjecture:globallylinked}
Let $d \geq 1$ be an integer and let $G = (V,E)$ be a graph. If a pair of vertices $\{u,v\}$ is globally linked in $G$ in $\RR^d$, then either $uv \in E$ or there is an $\mathcal{R}_d$-component $H = (V',E')$ of $G$ with $u,v \in V'$ and $\kappa(H,u,v) \geq d+1$.
\end{conjecture}

\begin{figure}[t]
    \centering
    \includegraphics[]{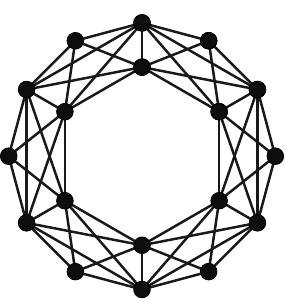}
        \caption{A graph $G$ that is $4$-connected and $\mathcal{R}_3$-connected, but not globally rigid in $\RR^3$ (see \cite{jordan.etal_2016}). It follows that some pair of vertices $\{u,v\}$ is not globally linked in $G$ in $\RR^3$, even though $\kappa(G,u,v) \geq 4$ holds.}
    \label{figure:6ring}
\end{figure}



An immediate consequence of \cref{theorem:jacksonjordanszabadka} is that ``being globally linked in $\RR^2$'' is a generic property in $\mathcal{R}_2$-connected graphs: for such a graph $G$ and a pair of vertices $u,v$ in $G$, $\{u,v\}$ is either globally linked in every generic realization of $G$ in $\RR^2$, or in none of them. The next example shows that this fails for $d = 3$.

Let $G$ be obtained from two copies of $K_5$ by gluing them along
three vertices. Let $S=\{s_1,s_2,s_3\}$ denote the set of identified vertices in $G$.
Now $S$ is a separator of size three in $G$ for which $G-S$ has two components
whose vertices we label $\{a_1,a_2\}$ and $\{b_1,b_2\}$, respectively.
Since $K_5$ is an $\mathcal{R}_3$-circuit, $G$ is $\mathcal{R}_3$-connected. Furthermore, every generic realization
$(G,p)$ in $\RR^3$ has exactly one other equivalent realization $(G,q)$ up to congruence, obtained by a partial
reflection corresponding to the $3$-fragment $\{a_1,a_2\}$.

\begin{figure}[t]
        \centering
\includegraphics[]{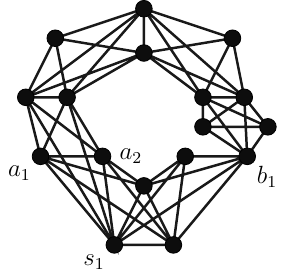}
        \caption{A graph that is $\mathcal{R}_3$-connected (and rigid in $\RR^3$) but in which ``being globally linked in $\RR^3$'' is not a generic property. By cutting along $a_1,a_2,s_1$ and $b_1$ we obtain the graphs $H$ and $G$, where $G$ is \mainclass[3] and $H$ is a chain of four copies of $K_5$, and thus it has three degrees of freedom in $\RR^3$.}
        \label{figure:Mconnectedexample}
\end{figure}

Let $H$ be the graph obtained from four copies of $K_5$, denoted by $H_1,\ldots,H_4$, by gluing $H_i$ and $H_{i+1}$ along two vertices for $i \in \{1,2,3\}$
so that the three vertex pairs of the resulting ``hinges'' are pairwise disjoint.
Since $K_5$ is an $\mathcal{R}_3$-circuit, $H$ is $\mathcal{R}_3$-connected, and furthermore it has three degrees of freedom.
Let $\{x,y,z\}$ be the three vertices of $H_1$ that are disjoint from the hinges and
let $w$ be a vertex in $H_4$ that is  
also disjoint from the hinges.
The key observation is the following. It can be shown using standard techniques from differential topology and the observation that the graph $H' = H+xw+yw+zw$ is rigid in $\RR^3$ (and in particular $r_3(H') = r_3(H) + 3$). 
\begin{lemma}\label{lemma:openball}
    Let $(H,p)$ be a generic realization in $\RR^3$ of the graph $H$ described above. There is a positive number $\varepsilon$ (depending on $(H,p)$) such that for every point $q_w \in \RR^3$ with $\norm{p(w) - q_w} < \varepsilon$ there is a realization $(H,q)$, equivalent to $(H,p)$, with $q(w) = q_w$ and $q(x) = p(x),q(y)=p(y),q(z)=p(z)$.
\end{lemma}

Let $K$ be obtained from $G$ and $H$ by identifying the vertices $a_1,a_2, s_1$ and $b_1$ with $x,y,z$ and $w$, respectively. We shall label the identified vertices according to the labeling in $G$. See \cref{figure:Mconnectedexample}.
Since $G$ and $H$ are both $\mathcal{R}_3$-connected, so is $K$. 
We shall construct two generic realizations $(K,p)$ and $(K,p')$ such that $\{a_1,b_1\}$ is globally linked in $(K,p)$, but not in $(K,p')$. 

Let us first consider an arbitrary generic realization $(H,p_H)$ of $H$ in $\RR^3$, and fix a positive number $\varepsilon$ with the properties described by \cref{lemma:openball}. Let us extend $(H,p_H)$ to a generic realization $(K,p)$ in such a way that the distance of $p(b_1)$ from the plane determined by $\{p(s_1),p(s_2),p(s_3)\}$ is less than $1/2\varepsilon$, and let $p_G$ denote the restriction of $p$ to the vertex set of $G$. We claim that $\{a_1,b_1\}$ is not globally linked in $(K,p)$. Indeed, we can take the framework $(G,q_G)$ obtained by the partial
reflection in the plane spanned by $\{p(s_1),p(s_2),p(s_3)\}$. By construction, we have $\norm{p_G(b_1) - q_G(b_1)} < \varepsilon$, so \cref{lemma:openball} guarantees the existence of a framework $(H,q_H)$ that is equivalent to $(H,p_H)$ and agrees with $(G,q_G)$ on the set of identified vertices. By identifying the common vertices of $(G,q_G)$ and $(H,q_H)$, we obtain a framework $(K,q)$ that is equivalent to $(K,p)$ but $\norm{q(a_1) - q(b_1)} \neq \norm{p(a_1) - p(b_1)}$, as claimed. 

Now let us fix an $a_1,b_1$-path $P$ in $K$ whose vertices are contained in $V(H) - \{s_1\}$. Consider a generic framework $(K,p')$ in $\RR^3$. Let $\delta$ denote the sum of the Euclidean lengths in $(K,p')$ of the edges in $P$, and let $\nu$ denote the distance of $p'(b_1)$ from the plane determined by $\{p'(s_1),p'(s_2),p'(s_3)\}$. Note that by (repeated applications of) the triangle inequality, in every framework $(K,q')$ that is equivalent to $(K,p')$ and for which the two frameworks agree on $\{a_1,a_2,s_1\}$, we have $\norm{p'(b_1) - q'(b_1)} \leq 2\delta$. On the other hand, if $\norm{q'(a_1) - q'(b_1)} \neq \norm{p'(a_1) - p'(b_1)}$ for such a realization $(K,q')$, then the restriction of $(K,q')$ to $G$ must necessarily be the framework obtained from the restriction of $(K,p')$ to $G$ by a partial reflection in the plane spanned by $\{p'(s_1),p'(s_2),p'(s_3)\}$. In particular, $\norm{p'(b_1) - q'(b_1)} = 2\nu$. It follows that if we choose $(K,p')$ in such a way that $\delta < \nu$ (which is clearly possible),
then $\{a_1,b_1\}$ will be globally linked in $(K,p')$, as desired.

We note that this example also shows that a graph obtained by gluing two graphs $G$ and $H$ that satisfy \cref{eq:kappa} need not satisfy \cref{eq:kappa}, even if we glue along at least $d+1$ vertices. This is in contrast with \cref{theorem:gluingkappa}, which shows that if $G$ and $H$ are also rigid, then the resulting graph does satisfy \cref{eq:kappa}.

\subsection{Globally linked pairs and gluing}

A key step in the proof of \cref{theorem:jacksonjordanszabadka} in \cite{jackson.etal_2006} is the following result, which implies \cref{theorem:linkedgluing} below.

\begin{theorem}\label{theorem:Mcircuit2separator}\cite[Corollary 5.5]{jackson.etal_2006} Let $G$ be an $\mathcal{R}_2$-circuit and let $\{u,v\}$ be a $2$-separator of $G$. Then $\{u,v\}$ is globally linked in $G$ in $\RR^2$.
\end{theorem}

\begin{theorem}\label{theorem:linkedgluing} Let $G = (V,E)$ be the union of the graphs $G_1 = (V_1,E_1), G_2 = (V_2,E_2)$ with $V_1 \cap V_2 = \{u,v\}$. If $\{u,v\}$ is linked in $G_i$ in $\RR^2$ for $i \in \{1,2\}$, then $\{u,v\}$ is globally linked in $G$ in $\RR^2$.
\end{theorem}
\begin{proof}
If $u$ and $v$ are adjacent in $G$, then the conclusion is immediate, so we may suppose that $u$ and $v$ are nonadjacent. Since $\{u,v\}$ is linked in $G_i$ in $\RR^2$, $G_i + uv$ contains an $\mathcal{R}_2$-circuit $C_i$ that spans $u$ and $v$, for $i \in \{1,2\}$. It is well-known that in this case $C = C_1 \cup C_2 - uv$ (the \emph{2-sum} of $C_1$ and $C_2$) is also an $\mathcal{R}_2$-circuit (see, e.g., \cite[Lemma 5.2]{jackson.etal_2006}). Since $\{u,v\}$ is a $2$-separator in $C$, \cref{theorem:Mcircuit2separator} implies that $\{u,v\}$ is globally linked in $C$ in $\RR^2$, and thus it is also globally linked in $G$ in $\RR^2$. 
\end{proof}

There is a similarly short proof of \cref{theorem:Mcircuit2separator} that uses \cref{theorem:linkedgluing}; in this sense, the two theorems are equivalent. 
We conjecture that this ``gluing theorem'' can be extended in two directions: one, to higher dimensions, and two, to graphs that are glued along larger vertex sets.

\begin{conjecture}\label{conjecture:gluing}
Let $G$ be the union of the graphs $G_1 = (V_1,E_1), G_2 = (V_2,E_2)$ with $|V_1 \cap V_2| \leq d+1$, and let $u,v \in V_1 \cap V_2$. If $\{u,v\}$ is linked in $G_i$ in $\RR^d$ for $i \in \{1,2\}$, then $\{u,v\}$ is globally linked in $G$ in $\RR^d$.
\end{conjecture}
\cref{theorem:gluingkappa} implies that \cref{conjecture:gluing} is true if $G_1$ and $G_2$ are both \mainclass[d]. The $d = 1$ case is also true:  this follows from the fact that a pair of vertices $\{u,v\}$ is linked in $G$ in $\RR^1$ if and only if $\kappa(G,u,v) \geq 1$, and globally linked in $G$ in $\RR^1$ if and only if $\kappa(G,u,v) \geq 2$.
We can also show that the conjecture holds when $d = 2$. Finally, in \cite{garamvolgyi_2023} the first author recently gave an affirmative answer for \cref{conjecture:gluing} in the special case when every vertex pair in $V_1 \cap V_2$ is linked in both $G_1$ and $G_2$ in $\RR^d$. We note that if we allow intersections of size larger than $d+1$, then the conclusion in \cref{conjecture:gluing} is not true in general. For the simplest counterexample, let $G = G_1 = G_2$ be a path of length $2$. In this case, the endvertices are linked, but not globally linked, in $G$ in $\RR^1$.

\subsection{Algorithmic aspects}

As we noted before, giving a combinatorial characterization (and a polynomial-time recognition algorithm) for (globally) rigid graphs in $\RR^d$ is a major open problem for $d \geq 3$, although probabilistic polynomial-time algorithms are known both in the case of rigidity \cite{asimow.roth_1978} and global rigidity \cite{gortler.etal_2010}. The situation is even worse in the case of globally linked pairs: no polynomial-time algorithm (probabilistic or otherwise) is known for testing whether a pair of vertices is globally linked in a graph in $\RR^d$, even for $d=2$.

For \mainclass[d] graphs, many questions related to global rigidity and globally linked pairs reduce to vertex connectivity problems, which are generally well-understood. In particular, we can check whether such a graph is globally rigid in $\RR^d$ by checking $(d+1)$-connectivity, and whether a pair of vertices is globally linked in the graph $\RR^d$ by checking local connectivity. We can also find an optimal augmentation into a globally rigid graph (i.e., a minimum cardinality set of edges whose addition makes the graph globally rigid in $\RR^d$) by finding an optimal augmentation into a $(d+1)$-connected graph, a problem for which there is an efficient algorithm, see \cite{vegh_2011}. We note that the optimal augmentation of a rigid graph into a globally rigid graph is known to be polynomial-time solvable in the $d = 1$ and $d=2$ cases, see \cite{kiraly.mihalyko_2022a}.

However, testing whether a graph is \mainclass[d] might be difficult even in the $d = 2$ case. (Recall that a graph on at least $2$ vertices is \mainclass[1] if and only if it is connected.) In this direction, we pose the following counterpart of \cref{theorem:linkedgluing} as a conjecture and describe how it leads to a polynomial-time algorithm for recognizing \mainclass[2] graphs.

\begin{conjecture}\label{conjecture:gluing2}
    Let $G = (V,E)$ be the union of the graphs $G_1 = (V_1,E_1), G_2 = (V_2,E_2)$ with $V_1 \cap V_2 = \{u,v\}$. If $\{u,v\}$ is not linked in $G_2$ in $\RR^d$, then $\{u,v\}$ is globally linked in $G$ in $\RR^d$ if and only if it is globally linked in $G_1$ in $\RR^d$.
\end{conjecture}

If \cref{conjecture:gluing2} is true, then we can decide in polynomial time whether a graph is \mainclass[2], as follows. Since the following discussion only concerns the $2$-dimensional case, we shall drop the suffix ``in $\RR^2$'' throughout the paragraph.  
Let $G = (V,E)$ be a graph on at least $3$ vertices. If $G$ is not rigid, then it is not \mainclass[2], while if it is $3$-connected, then it is \mainclass[2] if and only if it is globally rigid; each of these cases can be checked in polynomial time. If $G$ is rigid but not $3$-connected, then consider an inclusion-wise minimal $2$-fragment $X$, which can be found recursively using the fact that there are no crossing $2$-separators in $G$ (\cref{theorem:rigidnoncrossing}). Let us label the vertices in $N_G(X)$ as $\{u,v\}$ and consider $G_1 = G[X \cup N_G(X)] +uv$ and $G_2 = G[V-X] + uv$. 
Now \cref{theorem:closurechar} and the minimal choice of $X$ imply that $G$ is \mainclass[2] if and only if $G_1$ is globally rigid (which we can check in polynomial time), $G_2$ is \mainclass[2] (which we can check recursively), and $\{u,v\}$ is globally linked in $G$. 
Thus we only need to decide whether $\{u,v\}$ is globally linked in $G$, given that $G_1$ is globally rigid. If $G_1-uv$ is also globally rigid, then $\{u,v\}$ is globally linked in $G_1-uv,$ and thus also in $G$. Let us thus assume that $G_1 - uv$ is not globally rigid. It is still rigid by Hendrickson's Theorem \cite{hendrickson_1992}, so $\{u,v\}$ is linked, but not globally linked in $G$. It follows from \cref{theorem:linkedgluing} and (if true) the $2$-dimensional case of \cref{conjecture:gluing2} that $\{u,v\}$ is globally linked in $G$ if and only if it is linked in $G_2-uv$, which we can check in polynomial time.


In closing, we remark that given a polynomial-time algorithm for deciding whether a graph is globally rigid in $\RR^d$, it is also possible to check in polynomial time whether a graph has a globally rigid gluing construction in $\RR^d$, using a similar (but simpler) recursive approach as in the previous paragraph.

\section{Acknowledgements}
We thank Jan Legerský for drawing our attention to Darboux' porism and to the reference \cite{izmestiev_2020}. We also thank the anonymus referee for carefully reading the manuscript.

This work was supported by the Hungarian Scientific Research Fund provided by the National Research, Development and Innovation Office, grant Nos. K135421 and PD138102. 
The first author was supported by the ÚNKP-22-3 New National Excellence Program of the Ministry for Culture and Innovation from the source of the National Research, Development and Innovation Fund.
The second author was supported in part by the MTA-ELTE Momentum Matroid Optimization Research Group and the
National Research, Development and Innovation Fund of Hungary, financed under the ELTE
TKP 2021‐NKTA‐62 funding scheme.

\appendix
\crefalias{subsection}{appsec}
\section*{Appendix}
\renewcommand{\thesubsection}{\Alph{subsection}}

\subsection{Proofs for \texorpdfstring{\cref{subsection:fragments}}{Section 2.1}}
\label{appendix:fragments}

\begin{proof}[Proof of \cref{lemma:crossingcuts}]
The implications \textit{b)} $\Rightarrow$ \textit{c)} and \textit{d)} $\Rightarrow$ \textit{a)} are immediate from the definitions. By symmetry, it suffices to prove \textit{a)} $\Rightarrow$ \textit{b)}. Let $s,s' \in S$ be a pair of vertices lying in different components of $G-T$ and let $C$ be a component of $G-S$. Since both $s$ and $s'$ have a neighbour in $C$, there is a path in $G$ from $s$ to $s'$ in which the internal vertices all lie in $C$. Since $T$ separates $s$ and $s'$, it must contain an internal vertex of this path and thus $T \cap C$ is nonempty, as desired.
\end{proof}

\begin{proof}[Proof of \cref{minimumvertexcutorder}]
We need to show that $\preceq_{u,v}$ is reflexive, antisymmetric and transitive. Reflexivity is immediate. To show that $\preceq_{u,v}$ is antisymmetric, suppose that $S \preceq_{u,v} T \preceq_{u,v} S$ for some $(u,v)$-separating $d$-separators $S$ and $T$. The first relation implies that $T-S$ and $v$ lie in the same component of $G-S$, while from the second relation we obtain that $T-S$ and $u$ lie in the same component of $G-S$. Since $u$ and $v$ lie in different components of $G-S$ this can only happen if $T - S$ is empty, so that $T \subseteq S$. Since $|T| = |S|$, this implies $T = S$, as desired.

To show transitivity, suppose that $S \preceq_{u,v} T \preceq_{u,v} U$ for some $(u,v)$-separating $d$-separators $S,T,U$. This means that there are components $C_S$ and $C_U$ of $G-T$ such that $S - T \subseteq C_S \ni u$ and $U - T \subseteq C_U \ni v$. Every vertex of $T$ is adjacent to some vertex from $C_S$, so $(T - U) \cup C_S$ induces a connected subgraph in $G - U$. It follows that there is some component $C$ of $G-U$ for which $(T - U) \subseteq C$ and $u \in C_S \subseteq C$. This shows that $S \preceq_{u,v} U.$ 

If $S$ and $T$ are $(u,v)$-separating $d$-separators, then $S \preceq_{u,v} T$ implies that they are noncrossing. Finally, if $S$ and $T$ are noncrossing, then $T-S$ lies in a single component of $G-S$. This component must contain either $u$ or $v$, since otherwise $T$ could not be $(u,v)$-separating. It follows that either $S \preceq_{u,v} T$ or $T \preceq_{u,v} S$ holds, as required. 
\end{proof}

\begin{proof}[Proof of \cref{lemma:chainofvertexcuts}]
Since $S_1,\ldots,S_k$ form a descending chain, we have that $S_k - S_{k-1}$ lies in the same component of $G-S_{k-1}$ as $u$, while for $i < k-1$, $S_i - S_{k-1}$ lies in the same component of $G - {S_{k-1}}$ as $v$. In particular, $S_k - S_{k-1}$ and $\cup_{i=1}^{k-2} S_i - S_{k-1}$ are disjoint. It follows that $\varnothing \neq S_k - S_{k-1} = S_k - \cup_{i=1}^{k-1}S_i$, so the latter set is indeed nonempty.
\end{proof}

\begin{proof}[Proof of \cref{lemma:uxseparator}]
    Let $S = N_G(X)$ and $S' = N_G(X')$. Since $X$ and $X'$ are noncrossing, so are $S$ and $S'$, and hence we have $S' -S \subseteq C$ for some component $C$ of $G - S$. Since $S$ is a minimum vertex cut, every vertex in it has a neighbour in every component of $G-S$. It follows that $G - (S' \cup C)$ is connected, so there is some component $C'$ of $G - S'$ with $V - (S' \cup C) \subseteq C'$. By symmetry we also have $V - (S \cup C') \subseteq C$.
    
    We have $x \in C'$, and thus $u \notin C'$, since $X'$ is $(u,x)$-separating. Since $C'$ contains every component of $G-S$ apart from $C$, we must have $u \in C$, and thus $v \notin C$, since $X$ is $(u,v)$-separating. But $C$ contains every component of $G-S'$ apart from $C'$, so we must have $v \in C'$. This shows that $X'$ is $(u,v)$-separating. 
\end{proof}

\begin{proof}[Proof of \cref{lemma:dblock}]
\textit{(a)} Let us fix $u,v \in X$. It is clear that $\kappa(G',u,v) \leq \kappa(G,u,v)$, so we only need to show that $\kappa(G,u,v) \leq \kappa(G',u,v)$. Let $k = \kappa(G,u,v)$ and let $P_1,\ldots,P_k$ be internally vertex-disjoint $u,v$-paths in $G$. Now if $V(P_i) \cap \overline{X}$ is nonempty for some $i \in \{1,\ldots,k\}$, then $P_i$ must contain two vertices $x,y \in N_G(X)$. Since $N_G(X)$ a clique, we can shortcut $P_i$ along $xy$ to obtain a shorter path. After making every possible shortcut of this form, the resulting paths $P_1,\ldots,P_k$ will lie in $G'$. This shows that $k \leq \kappa(G',u,v)$.

\textit{(b)} First consider a $d$-separator $S$ of $G$ that intersects $X$. If $S$ also intersected $\overline{X}$, then $S$ and $N_G(X)$ would be crossing, which is impossible since $N_G(X)$ is a clique. It follows that $S \subseteq X \cup N_G(X)$. Moreover, $\overline{X} \cup (N_G(X) - S)$ induces a connected subgraph of $G$, so it belongs to some component $C$ of $G-S$. If $C'$ is another component of $G-S$, then we have $C' \subseteq X$, so $C \cap V(G')$ and $C' \cap V(G')$ are both nonempty, which shows that $S$ is also a $d$-separator in $G'$, as required. Conversely, if $S$ is a $d$-separator of $G'$ with $S \neq N_G(X)$, then clearly $S \cap X$ is nonempty and $S$ is also a $d$-separator of $G$.

\textit{(c)} First consider a $(d+1)$-block $Y'$ of $G'$. If $Y = N_G(X)$, then we are done; otherwise there is at least one vertex $u' \in Y' \cap X$. We claim that in this case $Y'$ is a $(d+1)$-block of $G$. Note that for every $u,v \in Y'$, either $u$ and $v$ are adjacent in $G'$ (and thus in $G$ as well), or $d +1 \leq \kappa(G',u,v) \leq \kappa(G,u,v)$. Thus $Y'$ is contained in some $(d+1)$-block $Y$ of $G$. If $v \in \overline{X}$, then $\kappa(G,u',v) \leq d$, which shows that $x \notin Y$. If $v \in (X \cup N_G(X)) - Y'$, then there is some vertex $u \in Y'$ such that $u$ and $v$ are nonadjacent and $\kappa(G',u,v) \leq d$. It follows that at most one of $u$ and $v$ is in $N_G(X)$, and hence by part \textit{(a)} we have $\kappa(G,u,v) = \kappa(G',u,v) \leq d$, so $v \notin Y$. This shows that $Y = Y'$, so $Y'$ is indeed a $(d+1)$-block of $G$.

Now let $Y$ be a $(d+1)$-block of $G$ that intersects $X$, and let us fix $u' \in Y \cap X$. We must have $Y \subseteq X \cup N_G(X)$, since if $v \in \overline{X}$, then $\kappa(G,u',v) \leq d$. For every nonadjacent pair $u,v \in Y$, at most one of $u$ and $v$ can be in $N_G(X)$, and thus by part \textit{(a)} we have $\kappa(G',u,v) = \kappa(G,u,v) \geq d + 1$. It follows that $Y$ is contained in some $(d+1)$-block $Y'$ of $G'$. If $v \in (X \cup N_G(X)) - Y$, then there is some vertex $u \in Y$ such that $u$ and $v$ are nonadjacent and $\kappa(G,u,v) \leq d$. It follows that at most one of $u$ and $v$ is in $N_G(X)$, and hence by part \textit{(a)} we have $\kappa(G',u,v) = \kappa(G,u,v) \leq d$, so $v \notin Y'$. This shows that $Y' = Y$, so $Y$ is indeed a $(d+1)$-block of $G'$.
\end{proof} 

\begin{proof}[Proof of \cref{lemma:inclusionwiseminimalfragment}]
Let us use the notation $S = N_G(X)$. Since $S$ is a $d$-separator, we have $\kappa(G,u,v) \leq d$ for any pair of vertices $u \in X$ and $v \in \overline{X}$. Thus, we only need to show that $\kappa(G,u,v) \geq d+1$ holds for any nonadjacent pair $u,v \in X \cup S$. Suppose for a contradiction that some nonadjacent pair $u,v \in X \cup S$ is separated by some $d$-separator $S'$. By its minimal choice, $X$  must be a single component of $G - S$. Since $S'$ separates $u$ and $v$, it has a nonempty intersection with $X$. But by assumption $S$ and $S'$ are noncrossing, so we must have $S' - S \subseteq X$. Since $S$ is a minimum vertex cut, every vertex in it has a neighbour in each component of $G - S$. Hence $G - (S' \cup X)$ is connected, so there is some component $C'$ of $G - S'$ with $V - (S' \cup X) \subseteq C'$. But then the rest of the components of $G - S'$ are all properly contained in $X$, contradicting the minimal choice of $X$.
\end{proof}

\subsection{Proofs of \texorpdfstring{\cref{subsection:gluing}}{Section 3.3}}
\label{appendix:gluing}

\begin{proof}[Proof of \cref{lemma:compatiblefragmentchar}]
    First, let us suppose that $X$ and $Y$ are compatible. By possibly taking complements we may assume that there is a $d$-fragment $Z$ of $G$ such that $X = Z \cap V_1$ and $Y = Z \cap V_2$. It follows from the first condition that $N_{G_1}(X) \subseteq N_G(Z)$. Since both sets have cardinality $d$, we have $N_{G_1}(X) = N_G(Z)$, and by an analogous reasoning $N_{G_2}(Y) = N_G(Z)$. It is immediate from this that for any pair of vertices $u,v \in V_1 \cap V_2$, $X$ is $(u,v)$-separating in $G_1$ if and only if $Z$ is $(u,v)$-separating in $G$ if and only if $Y$ is $(u,v)$-separating in $G_2$. This proves necessity.

    For sufficiency, let us suppose that $N_{G_1}(X) = N_{G_2}(Y)$ and that for every pair of vertices $u,v \in V_1 \cap V_2$, $X$ is $(u,v)$-separating in $G_1$ if and only if $Y$ is $(u,v)$-separating in $G_2$. It follows that, by possibly replacing $Y$ with $\overline{Y}$, we have $X \cap (V_1 \cap V_2) = Y \cap (V_1 \cap V_2)$. This implies, by a routine inspection, that $N_G(X \cup Y) = N_{G_1}(X) \cup N_{G_2}(Y) = N_{G_1}(X)$. It follows that $Z = X \cup Y$ is a $d$-fragment of $G$. Moreover, we have $Z \cap V_1 = X$ and $Z \cap V_2 = Y$, so $X$ and $Y$ are compatible, as claimed.
\end{proof}

\begin{proof}[Proof of \cref{lemma:compatiblesequencechar}]
    Let $S_i = N_{G_1}(X_i), i \in \{1,\ldots,k\}$ and $T_j = N_{G_2}(Y_j), j \in \{1,\ldots,l\}$.
     Let us suppose, first, that there is some pair $u,v \in V_1 \cap V_2$ of vertices and some index $i \in \{1,\ldots,k\}$ such that $X_i$ is $(u,v)$-separating and $S_i \neq T_j$ for every $j \in \{1,\ldots,l\}$ for which $Y_j$ is $(u,v)$-separating. Among the vertex pairs having these properties let us choose $u$ and $v$ so that the number of $(u,v)$-separating fragments among $X_1,\ldots,X_k$ is minimized. We claim that with this choice, the pair $\{u,v\}$ strongly separates $F_1$ and $F_2$. 
     
     Suppose for contradiction that there are indices $i_0$ and $j_0$ such that $S_{i_0} = T_{j_0}$, where $S_{i_0}$ is $(u,v)$-separating in $G_1$ and $T_{j_0}$ is $(u,v)$-separating in $G_2$. By the choice of $i$ we have $i_0 \neq i$, so (since $F_1$ is a reduced sequence) $S_{i_0} \neq S_{i}$. Let us fix $x \in S_{i_0} - S_i$. Now $X_i$ is either $(u,x)$-separating or $(x,v)$-separating; by symmetry we may suppose that it is the former. We claim that the pair $\{u,x\}$ contradicts the minimal choice of $\{u,v\}$. Indeed, if $Y_j$ is $(u,x)$-separating for some $j \in \{1,\ldots,l\}$, then noting that $x \in T_{j_0}$ and applying \cref{lemma:uxseparator} to $Y_{j_0}$ and $Y_j$ gives that $Y_{j}$ is also $(u,v)$-separating. By the choice of $X_i$ we get that $S_i \neq T_j$. Moreover, by another application of \cref{lemma:uxseparator} we have that every $(u,x)$-separating fragment among $X_1,\ldots,X_k$ is also $(u,v)$-separating. This shows that there are fewer $(u,x)$-separating fragments among $X_1,\ldots,X_k$ than $(u,v)$-separating fragments, since $X_{i_0}$ is $(u,v)$-separating but not $(u,x)$-separating. Thus, $\{u,x\}$ is a better choice than $\{u,v\}$, a contradiction. This shows that $\{u,v\}$ strongly separates $F_1$ and $F_2$. 

    We can find a strongly separating vertex pair by the same argument if there is some pair $u,v \in V_1 \cap V_2$ and some index $j \in \{1,\ldots,l\}$ such that $Y_j$ is $(u,v)$-separating and $T_j \neq S_i$ for every $i \in \{1,\ldots,k\}$ for which $S_i$ is $(u,v)$-separating. Thus, we may assume that neither of these cases hold. In other words, for every pair $u,v \in V_1 \cap V_2$,  if $X_i$ is $(u,v)$-separating, then there is some $(u,v)$-separating $d$-fragment $Y_j$ such that $S_i = T_j$, and similarly, if $Y_j$ is $(u,v)$-separating, then there is some $(u,v)$-separating $d$-fragment $X_i$ such that $S_i = T_j$. We show that in this case $F_1$ and $F_2$ are compatible. 
    
    By symmetry we may assume that $k \leq l$. Let us reorder $X_1,\ldots,X_k$ in such a way that each of $X_1,\ldots,X_r$ separate some pair of vertices from $V_1 \cap V_2$, while $X_{r+1},\ldots,X_k$ do not, for some $r \in \{0, \ldots, k\}$. Let us fix $i \in \{1,\ldots,r\}$. Since $X_i$ is $(u',v')$-separating for some pair $u',v' \in V_1 \cap V_2$, there is some $d$-fragment $Y_j$ that is also $(u',v')$-separating and for which $S_i = T_j$; by reordering $Y_1,\ldots,Y_l$, we may assume that $i = j$. We claim that for any pair of vertices $u,v \in V_1 \cap V_2$, $X_i$ is $(u,v)$-separating if and only if $Y_i$ is $(u,v)$-separating. Indeed, if $X_i$ is $(u,v)$-separating, then by our assumption there is some $d$-fragment $Y_{j'}$ such that $Y_{j'}$ is $(u,v)$-separating and $S_i = T_{j'}$. Since $S_i = T_i$, we have $T_i = T_{j'}$, but $F_2$ is reduced, so $i = j'$, and hence $T_i$ is $(u,v)$-separating, as required. An analogous argument shows that if $Y_i$ is $(u,v)$-separating, then so is $X_i$. It follows from \cref{lemma:compatiblefragmentchar} that $X_i$ and $Y_i$ are compatible.

    Repeating this argument for each $i \in \{1,\ldots,r\}$ we obtain that $X_i$ and $Y_i$ are compatible for each such $i$ (after suitably reordering $Y_1,\ldots,Y_l$). It also follows that for $j \in \{r+1,\ldots,l\}$, $Y_{j}$ does not separate any pair of vertices from $V_1 \cap V_2$. Indeed, if $Y_j$ is $(u,v)$-separating, then there is some $(u,v)$-separating $d$-fragment $X_i$ with $S_i = T_j$. This implies $i \leq r$ and thus $S_i = T_i$, so since $F_2$ is reduced we have $j = i \leq r$. Since $X_{r+1},\ldots,X_k$ do not separate any pair of vertices from $V_1\cap V_2$, by possibly taking complements (in $G_1$) we may assume that each of them is disjoint from $V_2$. Similarly, by possibly taking complements (in $G_2$) we may assume that each of $Y_{r+1},\ldots,Y_l$ is disjoint from $V_1$. This shows that $F_1$ and $F_2$ are indeed compatible. 
\end{proof}

\begin{proof}[Proof of \cref{lemma:compatiblesequences}]
    Note that by \cref{lemma:equivnoncrossing} we may freely reorder $X_1,\ldots,X_k$  ($Y_1,\ldots,Y_l$, respectively) and by \cref{lemma:equivcomplement} we may take complements (in $G_1$ and $G_2$, respectively); in this way we obtain sequences of partial $d$-reflections that are equivalent to $F_1$ (resp.\ $F_2$), and thus these operations do not change the congruence class of $(G_1,F_1(p_1))$ (resp.\ $(G_2,F_2(p_2))$). Using this observation and the definition of compatible sequences we may assume that the following is satisfied, for some integer $r$ with $0 \leq r \leq \min(k,l)$.
    \begin{itemize}
        \item There exists a $d$-fragment $Z_i$ of $G$ such that $X_i = Z_i \cap V_1$ and $Y_i = Z_i \cap V_2$, for every $i \in \{1,\ldots,r\}$.
    \item $X_i$ is disjoint from $V_2$ for every $i \in \{r+1,\ldots,k\}$.
    \item $Y_j$ is disjoint from $V_1$ for every $j \in \{r+1,\ldots,l\}$.
    \end{itemize}
    In particular, for every $i \in \{r+1,\ldots,k\}$ we have $N_{G_1}(X_i) = N_G(X_i)$, and thus $X_i$ is a $d$-fragment in $G$ as well. Similarly, $Y_j$ is a $d$-fragment in $G$ for each $j \in \{r+1,\ldots,l\}$.
    
    We define a sequence of partial $d$-reflections $F$ of $G$ by $$F = (\sep{Z_1},\ldots,\sep{Z_r},\sep{X_{r+1}},\ldots,\sep{X_k},\sep{Y_{r+1}},\ldots,\sep{Y_l}).$$ As noted in \cref{lemma:compatiblefragmentchar}, we have $N_{G}(Z_i) = N_{G_1}(X_i) = N_{G_2}(Y_i)$, and hence $Z_i$ is $(u,v)$-separating in $G$ if and only if $X_i$ is $(u,v)$-separating in $G_1$, for every pair $u,v \in V_1$ and every $i \in \{1,\ldots,r\}$. It follows that $\sep{Z_i}$ acts on the points $\{p(v), v \in V_1\}$ in exactly the same way as $\sep{X_i}$. Moreover, for each $j \in \{r+1,\ldots,l\}$ we have that $\sep{Y_j}$ is not $(u,v)$-separating for any $u,v \in V_1$, so $\sep{Y_j}$ acts as a congruence on $\{p(v), v \in V_1\}$. This shows that the restriction of $(G,F(p))$ to $G_1$ is congruent to $(G_1,F_1(p_1))$, and by an analogous reasoning the restriction of $(G,F(p))$ to $G_2$ is congruent to $(G_2,F_2(p_2))$.
\end{proof}

Before giving the proof of \cref{lemma:stronglyseparating}, we set some notation. 
Recall the notation introduced in the statement of \cref{lemma:stronglyseparating}. In the proof we shall use the partial ordering $\preceq_{u,v}$ on $(u,v)$-separating $d$-separators of $G_1$ and $G_2$. To avoid ambiguity, we shall use the notation $\preceq_{u,v}^{G_1}$ and $\preceq_{u,v}^{G_2}$. For a $d$-dimensional framework $(G,q)$ we shall let $(G_1,q_1)$ and $(G_2,q_2)$ denote the respective subframeworks, and for a subset of vertices $X$ we shall use the notation $q(X) = \{q(x): x\in X\}$. Let $F_1$ and $F_2$ be reduced sequences of partial $d$-reflections of $G_1$ and $G_2$, respectively. We shall say that $(G,q)$ is \emph{doubly admissible} if $(G_1,q_1)$ is $F_1$-admissible and $(G_2,q_2)$ is $F_2$-admissible.
 Note that since $F_1$ and $F_2$ are both reduced, a framework $(G,q)$ is doubly admissible if and only if both $q(N_{G_1}(X_i))$ and $q(N_{G_2}(Y_j))$ are affinely independent for each $i \in \{1,\ldots,k\}$ and $j \in \{1,\ldots,l\}$ (see the discussion after \cref{corollary:reducedsequence}).
\begin{proof}[Proof of \cref{lemma:stronglyseparating}]
We proceed in three stages, as in the proof of \cref{lemma:mainsimple}. First, we simplify and reorder the sequences $F_1$ and $F_2$ in a way that will be convenient later. Then we use a genericity argument to show that if some generic framework contradicts the statement of the lemma, then for every doubly admissible framework $(G,q)$ in $\RR^d$ 
we have $\norm{F_1(q_1)(u) - F_1(q_1)(v)} = \norm{F_2(q_2)(u) - F_2(q_2)(v)}$. Finally, we show that this is impossible by constructing a specific (non-generic) framework for which equality does not hold.

By the definition of being strongly separating, there is at least one fragment among $X_1,\ldots,X_k,Y_1,\ldots,Y_l$ that is $(u,v)$-separating. By symmetry, we may assume that this fragment is among $X_1,\ldots,X_k$. 
We first show that we may assume that $X_1,\ldots,X_k$ are all $(u,v)$-separating in $G_1$. Indeed, if not, then by \cref{lemma:equivnoncrossing} we may reorder $\sep{X_1},\ldots,\sep{X_k}$ so that $X_1,\ldots,X_{k_0}$ are $(u,v)$-separating in $G_1$ and $X_{k_0+1},\ldots,X_k$ are not $(u,v)$-separating in $G_1$ for some $k_0 \in \{1, \ldots, k\}$. Consider $F_1' = (\sep{X_{k_0+1}},\ldots,\sep{X_k})$. Since none of the fragments in $F_1'$ are $(u,v)$-separating, the distance between $u$ and $v$ is equal in $(G_1,q_1)$ and $(G_1,F_1'(q_1))$ for any $F_1'$-admissible framework $(G_1,q_1)$. It follows that we may disregard these partial reflections altogether and assume that $X_1,\ldots,X_k$ are all $(u,v)$-separating in $G_1$. Similarly, we may assume that either $F_2$ is trivial or $Y_1,\ldots,Y_l$ are all $(u,v)$-separating in $G_2$. 

Let $S_i = N_{G_1}(X_i), i \in \{1,\ldots,k\}$ and $T_j = N_{G_2}(Y_j),j \in \{1,\ldots,l\}$
. Since $X_1,\ldots,X_k$ and $Y_1,\ldots,Y_l$ are all $(u,v)$-separating and $\{u,v\}$ strongly separates $F_1$ and $F_2$, we have $S_i \neq T_j$ for every $i \in \{1,\ldots,k\},j \in \{1,\ldots,l\}$.
We may assume, by using \cref{lemma:equivnoncrossing} again, that 
\[S_k \preceq^{G_1}_{u,v} \ldots \preceq^{G_1}_{u,v} S_1 \hspace{2em} \text{ and that } \hspace{2em} T_l \preceq^{G_2}_{u,v} \ldots \preceq^{G_2}_{u,v} T_1.\] Similarly, by \cref{lemma:equivcomplement} we can assume that $u \in X_i$ for $i \in \{1,\ldots,k\}$ and $u \in Y_j$ for $j \in \{1,\ldots,l\}$. Note that these assumptions ensure that for any pair of indices $1 \leq i < j \leq k$ we have $S_j \subseteq X_i \cup S_i$, so that  $\sep{X_i}$ fixes the image of $S_j$ in any $F_1$-admissible framework. Similarly $\sep{Y_i}$ fixes the image of $T_j$ in any $F_2$-admissible framework, for any pair of indices $1 \leq i < j \leq l$.

Suppose for a contradiction that 
\[\norm{F_1(p_1)(u) - F_1(p_1)(v)} = \norm{F_2(p_2)(u) - F_2(p_2)(v)}\] holds for some generic realization $(G,p)$. 
By squaring each side, they become rational maps with rational coefficients in the coordinates of $p$. By \cref{lemma:rationalmapsgeneric}, if equality holds for the generic configuration $p$, then it holds for any configuration that is in the domain of both sides. This implies that for any doubly admissible framework $(G,q)$ we have 
\[\norm{F_1(q_1)(u) - F_1(q_1)(v)} = \norm{F_2(q_2)(u) - F_2(q_2)(v)}.\]

We contradict this by constructing a counterexample. It follows from \cref{lemma:chainofvertexcuts} that $S_k - \cup_{i=1}^{k-1}S_i$ is nonempty; let us fix a vertex $x$ from it. We consider two cases: either $x$ appears in $T_j$ for some $j \in \{1,\ldots,l\}$, or it does not. We give a different construction in each case.

First, suppose that $x \notin T_j$ for every $j \in \{1,\ldots,l\}$. In this case our construction is essentially the same as in the proof of \cref{lemma:mainsimple}; see \cref{fig:case1}. We define a doubly admissible framework $(G,q)$ as follows. Define $q(u) = q(v) = (0,\ldots,0), q(x) = (1,0,\ldots,0),$ and let us choose the values $q(z), z \in V - \{u,v,x\}$ so that each point lies in the $x_1 = 0$ hyperplane but the placement is relatively generic.\footnote{Recall that by ``relatively generic'' we mean that the coordinates of $q$ that are unspecified form an algebraically independent set over $\QQ$.} Now for every $j \in \{1,\ldots,l\}$, $q(T_j)$ is affinely independent and contained in the $x_1 = 0$ hyperplane, so the affine hyperplane spanned by it is precisely the $x_1 = 0$ hyperplane. It follows that $F_2(q_2)(u) = F_2(q_2)(v) = (0,\ldots,0)$. Similarly, $q(S_i)$ is contained in the $x_1 = 0$ hyperplane for $i \in \{1,\ldots,k-1\}$, but $q(S_k)$ is not, and by the choice of coordinates the affine hyperplane spanned by $q(S_k)$ does not contain the origin. It follows that $F_1(q_1)(v) \neq (0,\ldots,0)$, while $F_1(q_1)(u) = (0,\ldots,0)$. This shows that \[\norm{F_1(q_1)(u) - F_1(q_1)(v)} \neq \norm{F_2(q_2)(u) - F_2(q_2)(v)},\]a contradiction.

Finally, let us assume that there is some index $j \in \{1,\ldots,l\}$ such that $x \in T_j$. Let $j_0$ be the maximal such index. In this case we cannot ensure that all but one of $q(S_i)$ and $q(T_j)$ lie in the $x_1 = 0$ hyperplane, so we give a different construction; see \cref{fig:case2}. We assumed that $S_k \neq T_{j_0}$, so there exists a vertex $w \in S_k - T_{j_0}$ (this also shows that in this case $d$ is necessarily at least $2$). For $j \in \{1,\ldots,l\}$, let $F_2^j$ denote the subsequence $(\sep{Y_1},\ldots,\sep{Y_j})$ of $F_2$ and let $F_2^0 = \text{id}$ be the trivial sequence. We define a doubly admissible framework $(G,q)$ as follows. First, let $q(u) = q(v) = q(w) = (0,\ldots,0)$ and $q(x) = (1,0,\ldots,0)$. Next, by \cref{lemma:chainofvertexcuts}  $T_{j_0} - \cup_{j < j_{0}}T_j$ is nonempty; let us fix a vertex $t^*$ from this set. Choose the values $q(t), t \in \cup_{j \leq j_0} T_j - \{x,w,t^*\}$ so that each point lies in the $x_1 = 0$ hyperplane but the placement is relatively generic. 
For placing $q(t^*)$, we consider three subcases. If $t^* = x$, then we have no choice in placing $q(t^*)$. However, note that in this case $x$ is not contained in $T_j$ for any $j<j_0$. It follows that $q(T_j)$ is contained in the $x_1 = 0$ hyperplane for each $j < j_0$, and consequently $F^{j_0 - 1}_2(q_2)(v) = (0,\ldots,0)$. Next, if $t^* \neq x$ but $F^{j_0 - 1}_2(q_2)(v) = (0,\ldots,0)$ still holds, then choose $q(t^*)$ in a relatively generic way in the $x_1 = 0$ hyperplane; in particular, $q(T_{j_0})$ then does not contain the origin in its affine span. Otherwise choose $q(t^*)$ in the $x_{1} = 0$ hyperplane in a relatively generic way and such that the distance of $F^{j_0 - 1}_2(q_2)(v)$ and the affine hyperplane spanned by $q(T_{j_0})$ is less than half of $\norm{F^{j_0 - 1}_2(q_2)(v)}$. We can ensure this by choosing $q(t^*)$ in such a way that the line defined by $q(t^*)$ and $q(x)$ is sufficiently close to $F^{j_0 - 1}_2(q_2)(v)$. In each of these cases we have that $F^{j_0}_2(q_2)(v) \neq (0,\ldots,0)$. Finally, choose the rest of the values $q(s)$ in the $x_1 = 0$ hyperplane in a relatively generic way. Now each of $q(S_i), i \in \{1,\ldots,k\}$ and $q(T_j), j \in \{1,\ldots,j\}$ is affinely independent, so $(G,q)$ is doubly admissible. Moreover, by the maximal choice of $j_0$, every set $q(T_j), j \in \{j_0 + 1,\ldots,l\}$ is contained in the $x_1 = 0$ hyperplane, and thus its affine span is precisely this hyperplane. It follows that $F_2(q_2)(v)$ is either $F^{j_0}_2(q_2)(v)$ or its reflection in this hyperplane, which shows that $F_2(q_2)(v) \neq (0,\ldots,0).$ On the other hand, each of $q(S_i), i \in \{1,\ldots,k\}$ contains the origin in its affine span, so $F_1$ does not move $q(v)$, and thus $F_1(q_1)(v) = F_1(q_1)(u) = F_2(q_2)(u) = (0,\ldots,0)$ holds. It follows that \[\norm{F_1(q_1)(u) - F_1(q_1)(v)} \neq \norm{F_2(q_2)(u) - F_2(q_2)(v)}.\] This contradiction finishes the proof.
\end{proof}

\begin{figure}[th]
    \centering
    \begin{subfigure}[b]{0.32\linewidth}
    \centering
                \includegraphics[]{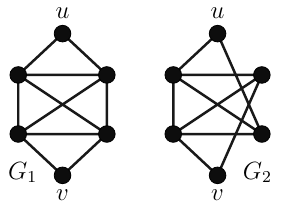}
        \caption{}
    \end{subfigure}
    \begin{subfigure}[b]{0.34\linewidth}
        \centering
                \includegraphics[]{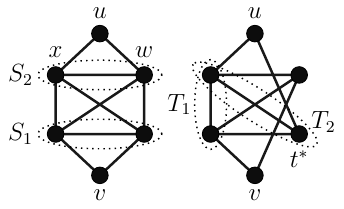}
        \caption{}
    \end{subfigure}
    \begin{subfigure}[b]{.31\linewidth}
    \centering
            \includegraphics[]{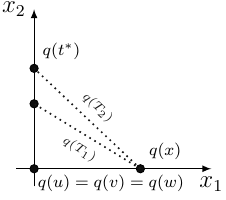}
    \caption{}
    \end{subfigure}
    \caption{The construction in the proof of \cref{lemma:stronglyseparating} in the case when $x \in T_j$ for some $j \in \{1,\ldots l\}$. \textit{(a)} Here $G = G_1 \cup G_2$, where $G_1$ and $G_2$ have the same vertex set. \textit{(b)} The minimum vertex cuts $S_1, S_2, T_1$ and $T_2$. Note that in general, it can also happen that $w \in T_j$ for some index $j$. \textit{(c)} The framework $(G,q)$ is constructed so that each of $q(S_1),\ldots,q(S_{k-1})$ lies in the $x_1 = 0$ hyperplane, and $q(S_k)$ contains the origin. In contrast, it can happen that the affine hyperplane spanned by $q(T_j)$ does not contain the origin for multiple choices of $j$. We place $q(t^*)$ in such a way as to ensure that $F_2^{j_0}$ moves $q(v)$ from the origin.}
    \label{fig:case2}
\end{figure}

\end{document}